\documentclass[reqno,11pt]{amsart}
\usepackage{amssymb}
\usepackage{amsmath}
\usepackage{amsthm}
\usepackage{mathtools}
\usepackage{mathrsfs}
\usepackage{graphicx}

\newtheorem{Th}{Theorem}[section]
\newtheorem{Lem}[Th]{Lemma}
\newtheorem{Prop}[Th]{Proposition}

\newtheorem{Def}[Th]{Definition}
\newtheorem{Prob}[Th]{Problem}
\newtheorem{Rem}[Th]{Remark}
\newtheorem{Assump}[Th]{Assumption}

\renewcommand{\theequation}{\arabic{section}.\arabic{equation}}

\newcommand{\R}{{\mathbb R}}

\newcommand{\N}{{\mathbb N}}
\newcommand{\Sym}{{\mathcal S}}

\newcommand{\ov}[1]{\overline{#1}}
\newcommand{\Tp}{{\mathrm T}}
\newcommand{\tr}{\operatorname{tr}}
\newcommand{\diag}{\operatorname{diag}}
\newcommand{\defeq}{\coloneqq}
\newcommand{\prz}[2]{ \frac{\partial{#1}}{\partial{#2}} }
\newcommand{\prd}[2]{ \frac{d{#1}}{d{#2}} }

\newcommand{\dvg}{\operatorname{div}}

\newcommand{\dt}{{\Delta t}}
\newcommand{\E}{{\mathcal{E}}}
\newcommand{\Lo}{{L^2(\Omega)}}

\newcommand{\bm}[1]{\boldsymbol{#1}}
\newcommand{\Hv}{{\bm{H}}}
\newcommand{\Vv}{{\bm{V}}}
\newcommand{\Hs}{{H}}

\newcommand{\blaket}[2]{ \left\langle {#1},{#2}\right\rangle_{\!\Vv^*,\Vv} }
\newcommand{\inner}[3]{ \left( {#1},{#2}\right)_{#3} }

\newcommand{\Pop}[1]{\mathcal{P}_{#1}}
\newcommand{\sigmas}[1]{{\sigma_{#1}^*}}

\newcommand{\tsigma}[1]{{p_{#1}}}
\newcommand{\eps}{\varepsilon}

\numberwithin{equation}{section}

\makeatletter
\newcommand*\bigcdot{\mathpalette\bigcdot@{.5}}
\newcommand*\bigcdot@[2]{\mathbin{\vcenter{\hbox{\scalebox{#2}{$\m@th#1\bullet$}}}}}
\makeatother

\makeatletter
\@namedef{subjclassname@2020}{\textup{2020} Mathematics Subject Classification}
\makeatother

\begin{document}
\title[Projection scheme for a perfect plasticity model]{Projection scheme for a perfect plasticity model with a time-dependent constraint set}

\author{Yoshiho Akagawa}
  \address[Yoshiho Akagawa]{National Institute of Technology (KOSEN), Gifu College, 2236-2 Motosu-shi, Gifu 501-0495, Japan.}
\email{akagawa@gifu-nct.ac.jp}

\author{Kazunori Matsui}
\address[Kazunori Matsui]{Department of Logistics and Information Engineering, Tokyo University of Marine Science and Technology, 2-1-6 Etchujima, Koto-ku,  Tokyo 135-0044, Japan.}
\email{kmat002@kaiyodai.ac.jp}
\urladdr{https://sites.google.com/site/kazunoriweb/home}



\begin{abstract}
  This paper introduces a new numerical scheme for a system that includes evolution equations describing a perfect plasticity model with a time-dependent yield surface. We demonstrate that the solution to the proposed scheme is stable under suitable norms. Moreover, the stability leads to the existence of an exact solution, and we also prove that the solution to the proposed scheme converges strongly to the exact solution under suitable norms.
\end{abstract}

\maketitle


\section{Introduction}
\label{sec:intro}

When a force is applied to materials like metals, they undergo elastic deformation and then shift to plastic deformation. This behavior is described by a relation between stress and strain. In engineering, the following simple model of perfect plasticity is often used \cite{DL76}:
\begin{align}\label{incl:perfect}
    \sigma = C(\E(u) - \eps_p),\qquad
    \prz{\eps_p}{t} \in \partial I_K(\sigma),
\end{align}
where
    $\sigma\in\Sym_d$ ($2 \le d \in\N$) is the stress,
    $\Sym_d$ is the space of symmetric matrices of order $d$,
    $C=(C_{ijkl})_{i,j,k,l}$ is the fourth-order elasticity tensor,
    $u\in\R^d$ is the displacement,
    $\E(u)\defeq(\nabla u + (\nabla u)^\Tp)/2\in\Sym_d$ is the strain,
    $\eps_p\in\Sym_d$ is the plastic part of $\E(u)$,
    $K\subset\Sym_d$ is a given closed convex set,
    $I_K$ is the indicator function on $K$,
    and $\partial I_K$ is the subdifferential of $I_K$.
The set $K$ is called the constraint set, and the boundary of $K$ is known as the yield surface.
In this paper, we consider the von Mises yield surface (yield criterion) \cite{DL76}:
\begin{align*}
    K \defeq \tilde{K}-p,\quad
    \tilde{K} \defeq \{\tau\in\Sym_d : |\tau^D|\le g\},
\end{align*}
where
    $\tau^D \defeq \tau-((\tr\tau)/d) E_d$ is the deviatoric part of $\tau$,
    $|\cdot|$ is the Frobenius norm for matrices,
    $E_d$ is the identity matrix of size $d$,
    and $p\in\Sym_d$ and $g\in\R$ are given.
This paper adopts the settings of \cite{AFK23}, where $g$ depends on time (and space) (cf. \cite{BBW15,NS22}). See Definition \ref{Def:Prob} for the details of $K$ and $g$. The equation \eqref{incl:perfect} is closely related to the Moreau sweeping process \cite{Mor71, Mor77}. The Moreau sweeping process is a problem of finding $\sigma:[0,T]\rightarrow H$ in a real Hilbert space $H$ with $T>0$, where $K(t)$ is a time-dependent closed convex constraint, and
\[
	\prd{\sigma}{t} \in - \partial I_{K}(\sigma)
    \quad \mbox{in }H
\]
is satisfied. This problem, modeling dynamic behavior under time-dependent constraints, has been studied in various scenarios by numerous researchers (refer to \cite{BKS06, KL02, KM09, KR11, KM98, Mon93, NS22, Rec11, Rec15, Vla13}).

\subsection{Problem}

Let $T>0$, and consider a bounded Lipschitz domain $\Omega$ in $\R^d$. We assume the existence of two subsets $\Gamma_1$ and $\Gamma_2$ of the boundary $\Gamma\defeq\partial\Omega$, with $|\Gamma_1|>0$ and $\Gamma_2 = \Gamma\setminus\Gamma_1$. Here, $|\Gamma_1|$ denotes the ($d-1$)-dimensional Hausdorff measure of $\Gamma_1$. In this paper, we focus on a problem incorporating Kelvin--Voigt viscosity, aiming to find the displacement $u:[0,T]\times\Omega\rightarrow\R^d$ and stress $\sigma_{\rm st}:[0,T]\times\Omega\rightarrow\Sym_d$ that satisfy 
\begin{align}\label{strong:original}\left\{\begin{aligned}
    \rho\frac{\partial^2 u}{\partial t^2} &= \dvg\sigma_{\rm st} + f &&\mbox{in }(0,T)\times\Omega,\\
    \sigma_{\rm st} &= \nu\E\left(\prz{u}{t}\right) + \sigma &&\mbox{in }(0,T)\times\Omega,\\
    \sigma &= C (\E(u)-\eps_p) &&\mbox{in }(0,T)\times\Omega,\\
    \prz{\eps_p}{t} &\in \partial I_K(\sigma) &&\mbox{in }(0,T)\times\Omega,\\
    u &=u_{\rm b} &&\mbox{on }(0,T)\times\Gamma_1,\\
    \sigma_{\rm st} n &=t_{\rm b} &&\mbox{on }(0,T)\times\Gamma_2,\\
    u(0,\cdot) &= u_0 &&\mbox{in }\Omega,\\
    \prz{u}{t}(0,\cdot) &= v_0 &&\mbox{in }\Omega,\\
    \sigma(0,\cdot) &= \sigma_0 &&\mbox{in }\Omega,\\
\end{aligned}\right.\end{align}
where
    $\rho>0$ is the density,
    $\nu>0$ is the viscosity coefficient,
    $\eps_p:(0,T)\times\Omega\rightarrow\Sym_d$ is the plastic part of $\E(u)$,
    $f:(0,T)\times\Omega\rightarrow\R^d$ is the external force,
    $v_{\rm b}:(0,T)\times\Gamma_1\rightarrow\R^d$ and 
    $\sigma_{\rm b}:(0,T)\times\Gamma_2\rightarrow\R^d$ are the boundary values, 
    $(u_0, v_0, \sigma_0)\in\R^d\times\R^d\times\Sym_d$ are the initial values with $\sigma_0 \in K(0)$,
and are given. Additionally, $n$ represents the outward unit normal vector to the boundary $\Gamma$.

The first equation in \eqref{strong:original} represents the motion equation, while the second, third, and fourth equations correspond to a rheological model depicted in Figure \ref{Fig:model}, which addresses small deformations. Equations five and six define boundary conditions on $\Gamma_1$ and $\Gamma_2$, respectively, while the seventh, eighth, and ninth equations set the initial conditions. The model depicted in Figure \ref{Fig:model} is used in engineering fields. It finds applications in areas such as construction materials \cite{SM16}, concrete flow analysis \cite{RGDTS07}, and concrete slump testing \cite{Roussel04}.

\begin{figure}[h]\centering
  \includegraphics[width=5cm]{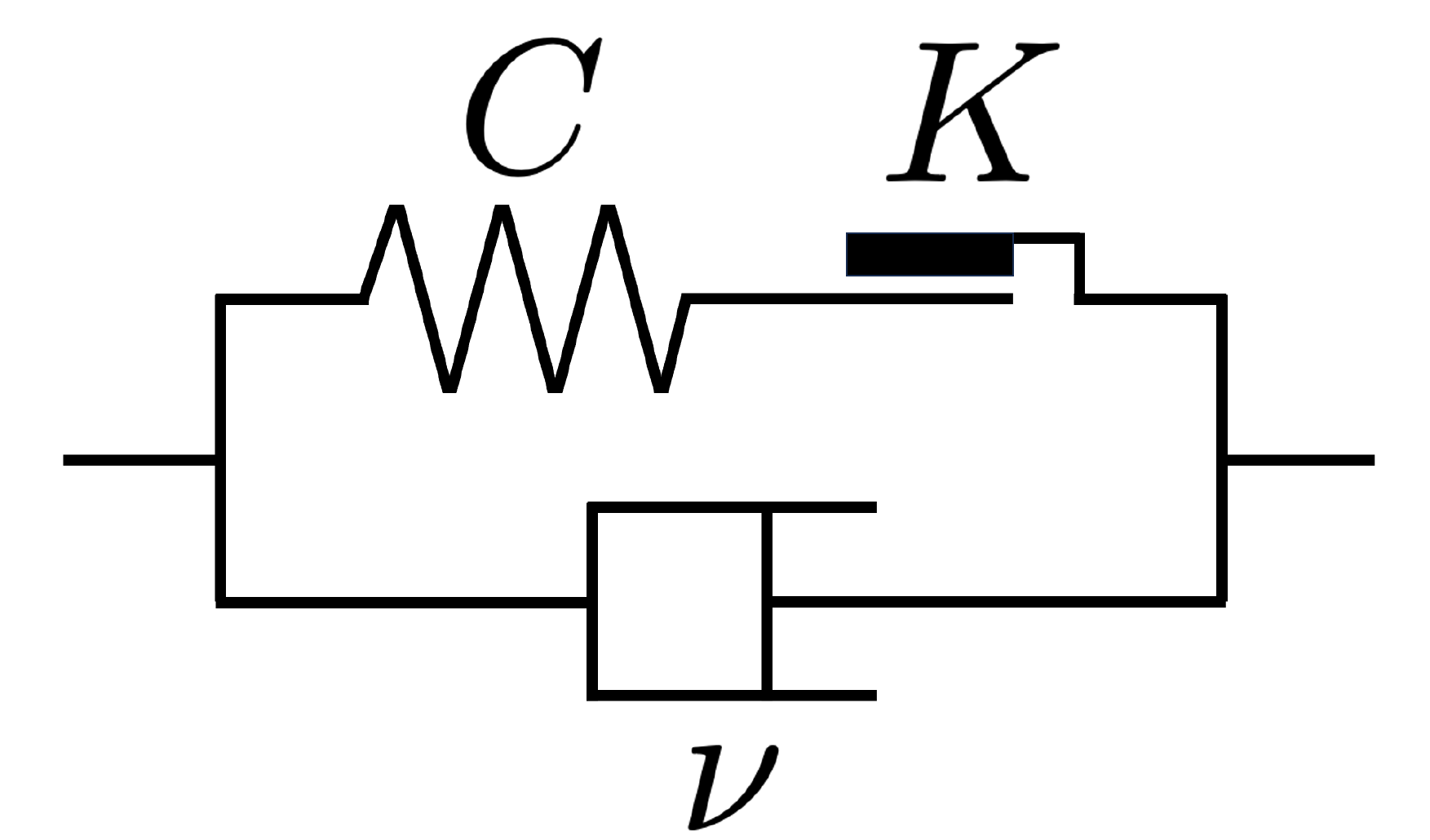}
  \caption{A schematic of the rheological model, showing Kelvin-Voigt and perfectly plastic elements arranged in parallel, used to derive the second, third, and fourth equations of \eqref{strong:original}.}\label{Fig:model}
\end{figure}

In subsequent discussions, we simplify by assuming $\rho=1$ and $C_{ijkl}=(\delta_{ik}\delta_{jl}+\delta_{il}\delta_{jk})/2$, where $\delta_{ij}$ is the Kronecker delta. 
By defining $v\defeq\partial u/\partial t$, obtaining $v$ also yields $u(t,x) = u_0(x) + \int_0^t v(s,x) ds$. Thus, solving the following problem suffices:

\begin{Prob}\label{Prob}
    Find $v:[0,T]\times\Omega\rightarrow\R^d$ and $\sigma:[0,T]\times\Omega\rightarrow\Sym_d$ that satisfy:
    \begin{align}\label{strong}\left\{\begin{aligned}
        \prz{v}{t} &= \nu\dvg\E(v) + \dvg\sigma + f &&\mbox{in }(0,T)\times\Omega,\\
        \prz{\sigma}{t} &\in \E(v) + h - \partial I_K(\sigma) &&\mbox{in }(0,T)\times\Omega,\\
        v &=0 &&\mbox{on }(0,T)\times\Gamma_1,\\
        (\nu\E(v) + \sigma) n &=0 &&\mbox{on }(0,T)\times\Gamma_2,\\
        v(0,\cdot) &= v_0 &&\mbox{in }\Omega,\\
        \sigma(0,\cdot) &= \sigma_0 &&\mbox{in }\Omega,\\
    \end{aligned}\right.\end{align}
    where
        $h:(0,T)\times\Omega\rightarrow\Sym_d$ is a function arising from converting non-homogeneous boundary conditions to homogeneous ones.
\end{Prob}

The problem with time-dependent threshold functions was proposed in \cite{AFK23}. 
For cases without time-dependent threshold functions, similar challenges incorporating heat transfer are addressed in \cite{KS98} (cf. \cite{BO18, BR13, CR19, Rossi18, BM17, Roubicek13}). In general, for solving elastoplastic problems numerically, it is necessary to employ nonlinear problem-solving algorithms, such as the Newton--Raphson method \cite{dSNPO08}. The scheme proposed in \cite{BR13} also employs the semismooth Newton method.

In this paper, we propose a new numerical scheme for Problem \ref{Prob}. The proposed scheme is stable under suitable norms regardless of the time step size and allows us to solve without the use of nonlinear problem-solving algorithms. The solutions of the proposed scheme satisfy the yield criterion for each time step. Furthermore, using this stability, we can show the existence of an exact solution (in the sense of Definition \ref{Def:Prob}). While spatial continuity and a positive lower bound of $g$ are assumed due to technical reasons for obtaining a well-posedness of Problem \ref{Prob} in \cite{AFK23}, we obtain the existence and uniqueness of the solution of Problem \ref{Prob} without these assumptions. Additionally, we also show that the solutions of the proposed scheme strongly converge to the exact solutions under suitable norms.

This paper is organized as follows. Section \ref{sec:Notation} defines the notation and the solutions to Problem \ref{Prob}. In Section \ref{sec:Numer}, we discretize Problem \ref{Prob} in the time direction and discuss the challenges of naively solving it either explicitly or implicitly, before introducing our proposed scheme. The main results are compiled in Section \ref{sec:Main}. Proofs of the main results are detailed in Section \ref{sec:Proofs}. Finally, Section \ref{sec:Conclusion} contains our conclusions.

\section{Notations and definition of a solution of Problem \ref{Prob}}\label{sec:Notation}

\subsection{Notations}

For a Banach space $X$, the dual pairing between $X$ and the dual space $X^*$ is denoted by $\langle\cdot,\cdot\rangle_{\!X^*,X}$, and we simply write $L^2(X)$, $L^\infty(X)$, and $H^1(X)$ as $L^2(0,T;X)$, $L^\infty(0,T;X)$, and $H^1(0,T;X)$, respectively.
We say that a function $u:[0, T]\rightarrow X$ is weakly continuous if, for all $f \in X^*$, the function defined by $[0,T]\ni t\mapsto\langle f,u(t)\rangle_{\!X^*,X}\in\R$ is continuous. We denote by $C^0([0,T];X_w)$ the set of functions defined on $[0,T]$ with values in $X$ which are weakly continuous.
For two sequences $(x_k)_{k=0}^N$ and $(y_k)_{k=1}^N$ in $X$, we define a piecewise linear interpolant $\hat{x}_\dt \in W^{1, \infty}(0, T; X)$ of $(x_k)_{k=0}^N$ and a piecewise constant interpolant $\bar{y}_\dt \in L^\infty(0, T; X)$ of $(y_k)_{k=1}^N$, respectively, by
\[\begin{aligned}
    \hat{x}_\dt(t)&:=x_{k-1}+\frac{t-t_{k-1}}{\dt}(x_k-x_{k-1})
    &&\text{for }t\in[t_{k-1},t_k] \mbox{ and }k=1,2,\ldots,N,\\
    \bar{y}_\dt(t)&:=y_k
    &&\text{for }t\in(t_{k-1},t_k] \mbox{ and }k=1,2,\ldots,N.
\end{aligned}\]
We define a backward difference operator by 
\[
    D_\dt x_k:=\frac{x_k-x_{k-1}}{\dt},\qquad
    D_\dt y_l:=\frac{y_l-y_{l-1}}{\dt}
\]
for $k=1,2,\ldots,N$ and $l=2,3,\ldots,N$.
Then, the sequence $(D_\dt x)_k:=D_\dt x_k$ satisfies 
$\frac{\partial\hat{x}_\dt}{\partial t}=(\overline{D_\dt x})_\dt$
on $(t_{k-1},t_k)$ for all $k=1,2,\ldots,N$.

\subsection{Definition of a solution of Problem \ref{Prob}}

We define the function spaces
$\Hv \defeq L^2(\Omega;\R^d)$,
$\Hs \defeq L^2(\Omega;\Sym_d)$,
$\Vv \defeq \{\varphi\in H^1(\Omega;\R^d):\varphi=0\mbox{ on }\Gamma_1\}$,
and let $\Vv^*$ be the dual space of $\Vv$. The solution to Problem \ref{Prob} is defined by using the following variational inequality:

\begin{Def}\label{Def:Prob}
    Given $\nu>0$, $v_0\in\Hv$, $\sigma_0\in\Hs$, $f\in L^2(0,T;\Vv^*)$, $h\in L^2(0,T;\Hs)$, $p\in H^1(0,T;\Hs)$, $g\in H^1(0,T;\Lo)$, and assume that for almost every $t\in[0,T]$ and almost every $x\in\Omega$, $g(t,x)\ge0$.
    We call the pair $(v, \sigma)\in (H^1(0,T;\Vv^*)\cap L^2(0,T;\Vv)) \times H^1(0,T;\Hs)$ a solution to Problem \ref{Prob} if:
    $v(0)=v_0$, $\sigma(0)=\sigma_0$ and for all $t\in[0,T]$,
    \[
        \sigma(t) \in K(t)
    \]
    is satisfied, and for almost every $t\in(0,T)$ and all $\varphi\in\Vv$ and $\tau\in K(t)$
    \begin{align}\label{P}\left\{\begin{aligned}
        &\blaket{\prd{v}{t}(t)}{\varphi}
        + \nu (\E(v(t)),\E(\varphi))_\Hs + (\sigma(t),\E(\varphi))_\Hs
        = \blaket{f(t)}{\varphi},\\
        &\left(\prd{\sigma}{t}(t)-\E(v(t)),\sigma(t)-\tau\right)_\Hs
        \le \left(h(t),\sigma(t)-\tau\right)_\Hs
    \end{aligned}\right.\end{align}
    holds. Here, $K(t)$ is a time-dependent function space defined as:
    \begin{align*}
        K(t) \defeq \tilde{K}(t)-p(t),\quad
        \tilde{K}(t) \defeq \left\{\tau\in\Hs : |\tau^D|\le g(t)
        \mbox{ a.e. in }\Omega\right\}.
    \end{align*}
\end{Def}

The discussions in \cite{AFK23} focus on solvability and parameter dependency in Problem \ref{Prob}, with the following proven results:

\begin{Th}[\cite{AFK23}]\label{Th:AFK}
    Under the following three conditions, there exists a unique solution to Problem \ref{Prob}.
    \begin{enumerate}
        \item[(A1)] $f \in L^2(0,T;\Hv)$, 
        \item[(A2)] $g \in H^1(0,T;C(\ov{\Omega}))$,
        \item[(A3)] there exists a constant $C>0$ such that
        \[
        0 < C \le g(t,x)\quad\mbox{for all }(t,x)\in[0,T]\times\ov{\Omega}.
        \]
    \end{enumerate}
\end{Th}

\section{Numerical scheme}\label{sec:Numer}

\subsection{Time discretization}

We consider numerical methods for solving Problem \ref{Prob}. Let $\dt=T/N$ ($N\in\N$) and let 
\[
    f_n \defeq \frac{1}{\dt}\int_{t_{n-1}}^{t_n} fdt, ~
    h_n \defeq \frac{1}{\dt}\int_{t_{n-1}}^{t_n} hdt, ~
    p_n \defeq p(t_n), ~
    g_n \defeq g(t_n), ~
    K_n \defeq K(t_n)
\]
for all $n=1,2,\ldots,N$, where $t_n=n\dt$.
Formally discretizing Problem \ref{Prob} implicitly yields the following.
\begin{align}\label{implicit}\left\{\begin{aligned}
    &\inner{\frac{v_n-v_{n-1}}{\dt}}{\varphi}{\Hv}
    + \nu (\E(v_n),\E(\varphi))_\Hs + (\sigma_n,\E(\varphi))_\Hs
    = \blaket{f_n}{\varphi},\\
    &\sigma_n + \tsigma{n}
    = \Pop{g_n}(\sigma_{n-1}+\dt(\E(v_n)+h_n)+\tsigma{n})
    \quad\mbox{in }\Hs,\\
\end{aligned}\right.\end{align}
for all $\varphi\in\Vv$, where $\Pop{R}:\Sym_d\rightarrow\Sym_d$ is defined for $A\in\Sym_d$, 
\begin{align*}
    \Pop{R}(A) \defeq \left\{\begin{aligned}
        &\frac{\tr A}{d}E_d + R\Phi\left(\frac{A^D}{R}\right) && \mbox{if }R>0,\\
        &\frac{\tr A}{d}E_d &&\mbox{if }R=0.
    \end{aligned}\right.\quad 
    \Phi(A) \defeq \left\{\begin{aligned}
        &A && \mbox{if }|A|\le 1,\\
        &\frac{A}{|A|} &&\mbox{if }|A|>1.
    \end{aligned}\right.
\end{align*}
Here, we remark that the second equation of \eqref{implicit} is equivalent to 
\begin{align}\label{ineq:2ndP}
    \left(\frac{\sigma_n - \sigma_{n-1}}{\dt}-\E(v_n),\sigma_n-\tau\right)_\Hs
    \le \left(h_n,\sigma_n-\tau\right)_\Hs
    \qquad\mbox{for all }\tau\in K_n.
\end{align}
See Theorem \ref{th:min} for the equivalence of \eqref{ineq:2ndP} and the second equation of \eqref{implicit}.

To solve \eqref{implicit}, it is necessary to use nonlinear problem-solving algorithms, such as the Newton--Raphson method or the semismooth Newton method, at each step to solve the nonlinear problem \cite{dSNPO08, BR13}. 
While it is possible to treat the third term of \eqref{implicit} as $(\sigma_{n-1},\E(\varphi))_\Hs$ explicitly, this leads to conditional stability and necessitates taking sufficiently small time steps.

\subsection{Proposed scheme}
 
We consider the following numerical scheme: Find $v_n\in\Vv$ and $\sigmas{n}, \sigma_n\in\Hs$ such that for all $\varphi\in\Vv$
\begin{align}\label{proj}\left\{\begin{aligned}
    &\inner{\frac{v_n-v_{n-1}}{\dt}}{\varphi}{\Hv}
    + \nu (\E(v_n),\E(\varphi))_\Hs + (\sigmas{n},\E(\varphi))_\Hs
    = \blaket{f_n}{\varphi} ,\\
    &\frac{\sigmas{n}-\sigma_{n-1}}{\dt} 
    = \E(v_n) + h_n
    \quad\mbox{in }\Hs,\\
    &\sigma_n + \tsigma{n}
    = \Pop{g_n}(\sigmas{n}+\tsigma{n})
    \quad\mbox{in }\Hs.\\
\end{aligned}\right.\end{align}

The first and second equations correspond to the discretization of the first and second equations of \eqref{strong}, with the $-\partial I_K(\sigma)$ term removed. Since $\sigmas{n} \notin K_n$ in general, the third equation involves projecting $\sigmas{n}$ onto the closed convex set $K_n$ to obtain $\sigma_n \in K_n$. This strategy is similar to the projection method used in numerical methods for incompressible viscous flow. In the projection method \cite{Chorin67,Temam69}, the velocity is solved without imposing the incompressibility condition, and then projected onto a divergence-free space to meet the incompressibility condition. As a projection method-like approach for hypo-elastoplastic problems, there is \cite{RSB15}, but this also involves projecting onto a linear space, similar to the projection method.

\section{Main results}\label{sec:Main}

\begin{Th}\label{Th:stab}
    (i) There exists a constant $c_1 > 0$ independent of $\dt$ such that for all $\dt\le1$,
    \[\begin{aligned}
        &\left\|\prd{\hat{v}_\dt}{t}\right\|_{L^2(\Vv^*)}^2
        + \|\bar{v}_\dt\|_{L^\infty(\Hv)}^2
        + \|\bar{v}_\dt\|_{L^2(\Vv)}^2
        + \frac{1}{\dt}\|\hat{v}_\dt - \bar{v}_\dt\|_{L^2(\Hv)}^2\\
        &+ \|\bar{\sigma}^*_\dt\|_{L^\infty(\Hs)}^2
        + \|\bar{\sigma}_\dt\|_{L^\infty(\Hs)}^2
        + \frac{1}{\dt}\|\bar{\sigma}_\dt - \bar{\sigma}_\dt^*\|_{L^2(\Hs)}^2\\
        \le~& c_1\left(\|v_0\|_\Hv^2 + \|\sigma_0\|_\Hs^2
        + \|f\|_{L^2(\Vv^*)}^2 
        + \|p\|_{H^1(\Hs)}^2
        + \|h\|_{L^2(\Hs)}^2\right).
    \end{aligned}\]
    In particular, $(\|\hat{v}_\dt\|_{H^1(\Vv^*)})_{0<\dt<1}$,
    $(\|\bar{v}_\dt\|_{L^\infty(\Hv)})_{0<\dt<1}$, 
    $(\|\bar{v}_\dt\|_{L^2(\Vv)})_{0<\dt<1}$, \\
    $(\|\bar{\sigma}^*_\dt\|_{L^\infty(\Hs)})_{0<\dt<1}$, and 
    $(\|\bar{\sigma}_\dt\|_{L^\infty(\Hs)})_{0<\dt<1}$ are bounded and 
    $\|\hat{v}_\dt - \bar{v}_\dt\|_{L^2(\Hv)}$ and $\|\bar{\sigma}_\dt - \bar{\sigma}_\dt^*\|_{L^2(\Hs)}$ converge to $0$ as $\dt\rightarrow0$.

    \noindent (ii) There exists a constant $c_2 > 0$ independent of $\dt$ such that for all $\dt\le1$,
    \[\begin{aligned}
        \|\hat{\sigma}_\dt\|_{H^1(\Hs)}
        \le& c_2(\|v_0\|_\Hv + \|\sigma_0\|_\Hs + \|f\|_{L^2(\Vv^*)} 
        + \|p\|_{H^1(\Hs)} + \|h\|_{L^2(\Hs)} + \|g\|_{H^1(\Lo)}).
    \end{aligned}\]
    In particular, $(\hat{\sigma}_\dt)_{0<\dt<1}$ is bounded in $H^1(\Hs)$.
\end{Th}

\begin{Rem}\label{Rem:weakstab}
    If $g$ is not in $H^1(\Lo)$ but in $C([0,T];\Lo)$, we can show that the sequence $(\hat{\sigma}_\dt)_{0<\dt<1}$ is equicontinuous.
\end{Rem}

From the boundedness obtained in Theorem \ref{Th:stab}, we obtain that the sequences \\$(\hat{v}_\dt)_{0<\dt<1}$, $(\bar{v}_\dt)_{0<\dt<1}$, $(\bar{\sigma}^*_\dt)_{0<\dt<1}$, $(\bar{\sigma}_\dt)_{0<\dt<1}$, and $(\hat{\sigma}_\dt)_{0<\dt<1}$ have subsequences that converge weakly. Since the resulting limit $(v,\sigma)$ is a solution to Problem \ref{Prob}, the following theorem can be established without assumptions (A1), (A2), and (A3) of Theorem \ref{Th:AFK}.

\begin{Th}\label{Th:exact}
    For all $\nu>0$, $f\in L^2(\Vv^*)$, $p \in H^1(\Hs)$, $v_0 \in \Hv$, $\sigma_0 \in K(0)$, $g \in H^1(\Lo)$ with $g\ge 0$ a.e on $\Omega$ for all $t\in[0,T]$, there exists a unique solution $(v, \sigma)\in (H^1(\Vv^*)\cap L^2(\Vv)) \times H^1(\Hs)$ to Problem \ref{Prob}.
\end{Th}

\begin{Rem}
    If $g$ is not in $H^1(\Lo)$ but in $C([0,T];\Lo)$, by Remark \ref{Rem:weakstab} and the Ascoli--Arzel{\`{a}} theorem, one shows existence (and uniqueness) of a weak solution $(v, \sigma) \in (H^1(\Vv^*)\cap L^2(\Vv)) \times C([0,T];\Hs)$ to Problem \ref{Prob} in the sense of \cite[Definition 2.2]{AFK23} using the same way as the proof of Theorem \ref{Th:exact}.
\end{Rem}

Furthermore, by assuming smoothness, one can show that the solutions to \eqref{proj} strongly converge to the solution to Problem \ref{Prob}. 

\begin{Th}\label{Th:strong}
    Under the assumption of Theorem \eqref{Th:exact}, if $v_0\in\Vv$ and $f\in L^2(\Hs)$, then we have
    \begin{align*}
        \hat{\sigma}_\dt &\rightarrow \sigma\quad\mbox{strongly in }L^\infty(\Hs),\\
        \hat{v}_\dt &\rightarrow v\quad\mbox{strongly in }L^\infty(\Hv),\\
        \bar{v}_\dt &\rightarrow v\quad\mbox{strongly in }L^2(\Vv),
    \end{align*}
    as $\dt\rightarrow 0$.
\end{Th}

\section{Proofs of main results}\label{sec:Proofs}

\subsection{Preliminary result}

We recall the Korn inequality and the discrete Gronwall inequality.

\begin{Lem}
    {\rm \cite[Lemma 5.4.18]{Trangenstein13}}
    There exists a constant $c_K>0$ such that 
    \[
        \frac{1}{c_K}\|\varphi\|_\Vv \le \|\E(\varphi)\|_\Hs 
        \le c_K\|\varphi\|_\Vv \quad\mbox{for all }\varphi\in \Vv.
    \]
\end{Lem}

\begin{Lem}\label{lem:gronwall}
    {\rm \cite[Lemma 5.1]{HR90}}
  Let $\dt,\beta>0$ and let non-negative sequences 
  $(a_k)^N_{k=0}$, $(b_k)^N_{k=0}$, $(c_k)^N_{k=0}$, 
  $(\alpha_k)^N_{k=0} \subset \{x\in\R : x \ge 0\}$ satisfy that
  \[
      a_n+\dt\sum^m_{k=0}b_k
      \le \dt\sum^m_{k=0}\alpha_k a_k+\dt\sum^m_{k=0}c_k+\beta
      \qquad\mbox{for all }m=0,1,\ldots,N.
  \]
  If $\alpha_k\dt<1$ for all $k=0,1,\ldots,N$, then we have
  \[
      a_n+\dt\sum^m_{k=0}b_k
      \le e^C\left(\dt\sum^m_{k=0}c_k+\beta\right)
      \qquad\mbox{for all }m=0,1,\ldots,N,
  \]
  where $C:=\dt\sum^N_{k=0}\frac{\alpha_k}{1-\alpha_k\dt}$.
\end{Lem}

We prepare the following lemma, which is derived using the Ascoli--Arzel{\`{a}} theorem.

\begin{Lem}\label{Lem:AA}
    Let $X$ be a Banach space such that $X^*$ is separable. If the sequence $(u_n)_{n\in\N}\subset C([0,T];X)$ satisfies the following two conditions:
    \begin{enumerate}
        \item[(i)] there exists a constant $c>0$ such that for all $n\in\N$ and $t\in[0,T]$, $\|u_n(t)\|_X < c$,
        \item[(ii)] $(u_n)$ is equicontinuous, i.e., for all $t\in[0,T]$ and $\eps>0$, there exists $\delta>0$ such that if $s\in[0,T]$ satisfies $|s-t|<\delta$, then $\|u_n(s) - u_n(t)\|_X < \eps$ for all $n\in\N$,
    \end{enumerate}
    then there exist a subsequence $(n_k)_{k\in\N}$ and $u\in C([0,T];X_w)$ such that for all $t\in[0,T]$,
    \[
        u_{n_k}(t) \rightharpoonup u(t) \mbox{ weakly in } X
    \]
    as $k\rightarrow\infty$.
\end{Lem}

\begin{proof}
Let $B\defeq\{v\in X:\|v\|_X\le c\}$. There exists a metric $d:B\times B\rightarrow\R$, which induces the weak topology on $X$ and satisfies $d(v,w)\le\|v-w\|_X$ for all $v,w\in B$ \cite[Theorem 3.29]{Bre11}. Given the conditions (i) and (ii), the sequence $(u_n)_{n\in\N}$ belongs to $C([0,T];(B,d))$ and is equicontinuous in the metric space $(B,d)$. Moreover, the metric space $(B,d)$ is compact. By the Ascoli--Arzel{\`{a}} theorem \cite[Theorem 4.25]{Bre11}, there exist a subsequence $(n_k)_{k\in\N}\subset\N$ and $u\in C([0,T];(B,d))$ such that for all $t\in[0,T]$,
\[
    d(u_{n_k}(t), u(t))\rightarrow 0
\]
as $k\rightarrow\infty$, leading to the conclusion.
\end{proof}

We show some properties of the function $\Pop{R},\Phi:\Sym_d\rightarrow\Sym_d$.

\begin{Prop}\label{Prop:Phi}
    \begin{enumerate}
        \item[(i)] It holds that for all $R\ge0$ and $A \in \Sym_d$,
        \[\begin{aligned}
            (E_d, A^D) 
            = \left(E_d, R\Phi\left(\frac{A^D}{R}\right)\right) = 0,\quad
            |\Pop{R}(A)|^2 
            = \left|\frac{\tr A}{d}E_d\right|^2 + \left|R\Phi\left(\frac{A^D}{R}\right)\right|^2.
        \end{aligned}\]
        \item[(ii)] It holds that for all $R_1, R_2\ge0$ and $A \in \Sym_d$,
        \[
            |\Pop{R_1}(A) - \Pop{R_2}(A)| \le |R_1 - R_2|.
        \]
        \item[(iii)] It holds that for all $R\ge0$ and $A, B \in \Sym_d$ with $|B^D|\le R$,
        \[
            (\Pop{R}(A) - A, \Pop{R}(A) - B) \le 0.
        \]
        \item[(iv)] It holds that for all $R\ge0$ and $A, B \in \Sym_d$,
        \[
            |\Pop{R}(A) - \Pop{R}(B)| \le |A - B|.
        \]
    \end{enumerate}
\end{Prop}

\begin{proof}
    (i) It holds that for all $R\ge0$ and $A \in \Sym_d$,
    \begin{align*}
        (E_d, A^D) &= \left(E_d, A - \frac{\tr A}{d}E_d\right)
        = \tr A - \tr A = 0,\\
        \left(E_d, R\Phi\left(\frac{A^D}{R}\right)\right) &= 0,
    \end{align*}
    and hence 
    \begin{align*}
        |\Pop{R}(A)|^2 
        &= \left|\frac{\tr A}{d}E_d\right|^2 
        + 2\frac{\tr A}{d}\left(E_d, R\Phi\left(\frac{A^D}{R}\right)\right)
        + \left|R\Phi\left(\frac{A^D}{R}\right)\right|^2\\
        &= \left|\frac{\tr A}{d}E_d\right|^2 
        + \left|R\Phi\left(\frac{A^D}{R}\right)\right|^2. 
    \end{align*}

    (ii) It holds that for all $R_1, R_2\ge0$ and $A \in \Sym_d$,
    \[
        |\mathcal{P}_{R_1}(A) - \mathcal{P}_{R_2}(A)|
        = \left|R_1\Phi\left(\frac{A^D}{R_1}\right) - R_2\Phi\left(\frac{A^D}{R_2}\right)\right|
        \le |R_1 - R_2|.
    \]
    
    (iii) If $A\in\Sym_d$ and $R\ge0$ satisfy the condition that $|A^D| \le R$, it follows that $\Pop{R}(A) = A$ and for all $B \in \Sym_d$,
    \[
        (\Pop{R}(A) - A, \Pop{R}(A) - B)
        = (A - A, \Pop{R}(A) - B) = 0.
    \]
    Thus, we consider the case where $|A^D| > R$. For all $B \in \Sym_d$ with $|B^D|\le R$,
    \[\begin{aligned}
        &(\Pop{R}(A) - A, \Pop{R}(A) - B)\\
        =& \left(R\Phi\left(\frac{A^D}{R}\right) - A^D, 
        \frac{\tr (A-B)}{d}E_d + R\Phi\left(\frac{A^D}{R}\right) - B^D\right)\\
        =& \left(R - |A^D|\right)\left(R - \frac{(A^D, B^D)}{|A^D|}\right)
        \le \left(R - |A^D|\right)\left(R - |B^D|\right)
        \le 0.
    \end{aligned}\]
    
    (iv) By (iii), it holds that for all $R\ge0$ and $A, B \in \Sym_d$,
    \[\begin{aligned}
        (\Pop{R}(A) - A, \Pop{R}(A) - \Pop{R}(B)) \le 0,\\
        (\Pop{R}(B) - B, \Pop{R}(B) - \Pop{R}(A)) \le 0,
    \end{aligned}\]
    and hence,
    $
        |\Pop{R}(A) - B| \le |A - B|.
    $
    By adding the two inequalities, we obtain
    \[
        |\Pop{R}(A) - \Pop{R}(B)|^2 
        \le (A - B, \Pop{R}(A) - \Pop{R}(B))
        \le |A - B||\Pop{R}(A) - \Pop{R}(B)|,
    \]
    which implies the conclusion.
\end{proof}

\subsection{Stability}

\begin{proof}[Proof of Theorem \ref{Th:stab}]
    (i) Putting $\varphi=v_n$ in the first equation of \eqref{proj}, we obtain 
    \begin{align*}
        \inner{\frac{v_n-v_{n-1}}{\dt}}{v_n}{\Hv}
        + \nu (\E(v_n),\E(v_n))_\Hs + (\sigmas{n},\E(v_n))_\Hs
        = \blaket{f_n}{v_n}.
    \end{align*}
    Applying the Korn inequality leads to
    \begin{align}\label{stab_ineq1}\begin{aligned}
        &\frac{1}{2\dt}\left(\|v_n\|_\Hv^2-\|v_{n-1}\|_\Hv^2+\|v_n-v_{n-1}\|_\Hv^2\right)
        + \nu\|\E(v_n)\|_\Hs^2 + (\sigmas{n},\E(v_n))_\Hs\\
        \le& \|f_n\|_{\Vv^*}\|v_n\|_\Vv
        \le c_K\|f_n\|_{\Vv^*}\|\E(v_n)\|_\Hs
        \le \frac{\nu}{4}\|\E(v_n)\|_\Hs^2
        + \frac{c_K^2}{\nu}\|f_n\|_{\Vv^*}^2.
    \end{aligned}\end{align}
    Multiplying the second equation of \eqref{proj} by $\sigmas{n}+\tsigma{n}$, we have
    \begin{align}\label{stab_ineq2}\begin{aligned}
        0=&\left(\frac{\sigmas{n}-\sigma_{n-1}}{\dt} - \E(v_n) - h_n,
        \sigmas{n}+\tsigma{n}\right)_\Hs \\
        %
        %
        \ge &\frac{1}{2\dt}(\|\sigmas{n}+\tsigma{n}\|_\Hs^2 
        - \|\sigma_{n-1}+\tsigma{n-1}\|_\Hs^2)
        - (\E(v_n),\sigmas{n})_\Hs - \|\E(v_n)\|_\Hs\|\tsigma{n}\|_\Hs\\
        &- \left(\left\|\frac{\tsigma{n}-\tsigma{n-1}}{\dt}\right\|_\Hs+\|h_n\|_\Hs\right)\|\sigmas{n}+\tsigma{n}\|_\Hs\\
        \ge &\frac{1}{2\dt}(\|\sigmas{n}+\tsigma{n}\|_\Hs^2 
        - \|\sigma_{n-1}+\tsigma{n-1}\|_\Hs^2) 
        - (\E(v_n),\sigmas{n})_\Hs - \frac{\nu}{4}\|\E(v_n)\|_\Hs^2\\
        &- \frac{1}{4}\|\sigmas{n}+\tsigma{n}\|_\Hs^2
        - \frac{1}{\nu}\|\tsigma{n}\|_\Hs^2
        - 2\left(\left\|\frac{\tsigma{n}-\tsigma{n-1}}{\dt}\right\|_\Hs^2+\|h_n\|_\Hs^2\right).
    \end{aligned}\end{align}
    The third equation of \eqref{proj} and Proposition \ref{Prop:Phi}(iii) imply that for all $n=2,\ldots,N+1$,
    \begin{align}\label{stab_ineq3}\begin{aligned}
        0&\ge \left((\sigma_{n-1}+\tsigma{n-1})-(\sigmas{n-1}+\tsigma{n-1}),\sigma_{n-1}+\tsigma{n-1}\right)_\Hs\\
        &= \frac{1}{2}(\|\sigma_{n-1}+\tsigma{n-1}\|_\Hs^2 - \|\sigmas{n-1}+\tsigma{n-1}\|_\Hs^2 + \|\sigma_{n-1}-\sigmas{n-1}\|_\Hs^2).
    \end{aligned}\end{align}
    If we set $\sigmas{0}\defeq\sigma_0$, then \eqref{stab_ineq3} holds for $n=1$.
    By \eqref{stab_ineq1} , \eqref{stab_ineq2}, and \eqref{stab_ineq3}, we have 
    \begin{align}\label{stab_ineq4}\begin{aligned}
        &\frac{1}{2\dt}\left(\|v_n\|_\Hv^2-\|v_{n-1}\|_\Hv^2+\|v_n-v_{n-1}\|_\Hv^2\right)\\
        &+ \frac{1}{2\dt}(\|\sigmas{n}+\tsigma{n}\|_\Hs^2
        - \|\sigmas{n-1}+\tsigma{n-1}\|_\Hs^2
        + \|\sigma_{n-1}-\sigmas{n-1}\|_\Hs^2) + \frac{\nu}{2}\|\E(v_n)\|_\Hs^2 \\
        \le& \frac{c_K^2}{\nu}\|f_n\|_{\Vv^*}^2 + \frac{1}{4}\|\sigmas{n}+\tsigma{n}\|_\Hs^2
        + \frac{1}{\nu}\|\tsigma{n}\|_\Hs^2
        + 2\left(\left\|\frac{\tsigma{n}-\tsigma{n-1}}{\dt}\right\|_\Hs^2+\|h_n\|_\Hs^2\right).
    \end{aligned}\end{align}
    By summing up \eqref{stab_ineq4} for $n = 1,2,\ldots,m$, where $m\le N$, we obtain
    \begin{align*}\begin{aligned}
        &\|v_m\|_\Hv^2 + \|\sigmas{m}+\tsigma{m}\|_\Hs^2 
        - \|v_0\|_\Hv^2 - \|\sigma_0+\tsigma{0}\|_\Hs^2\\
        &\quad + \sum_{n=1}^m\left(\|v_n-v_{n-1}\|_\Hv^2 + \|\sigma_{n-1}-\sigmas{n-1}\|_\Hs^2\right)
        + \nu\dt\sum_{n=1}^m\|\E(v_n)\|_\Hs^2\\
        \le& \frac{\dt}{2}\sum_{n=1}^m \|\sigmas{n}+\tsigma{n}\|_\Hs^2
        + c_1\dt\sum_{n=1}^m\left(\|f_n\|_{\Vv^*}^2 + \|p_n\|_\Hs^2 + \left\|\frac{\tsigma{n}-\tsigma{n-1}}{\dt}\right\|_\Hs^2 + \|h_n\|_\Hs^2\right),
    \end{aligned}\end{align*}
    where $c_1\defeq\max\{2c_K^2/\nu, 2/\nu, 4\}$.
    By the discrete Gronwall inequality, if $\dt\le1$, then it holds that for all $m=1,2,\ldots,N$,
    \begin{align*}\begin{aligned}
        &\|v_m\|_\Hv^2 + \|\sigmas{m}+\tsigma{m}\|_\Hs^2
        + \sum_{n=1}^m\left(\|v_n-v_{n-1}\|_\Hv^2 + \|\sigma_{n-1}-\sigmas{n-1}\|_\Hs^2\right)
        + \nu\dt\sum_{n=1}^m\|\E(v_n)\|_\Hs^2\\
        \le& e\Bigg(\|v_0\|_\Hv^2 + \|\sigma_0+\tsigma{0}\|_\Hs^2
        + c_1\dt\sum_{n=1}^m\left(\|f_n\|_{\Vv^*}^2 + \|p_n\|_\Hs^2 + \left\|\frac{\tsigma{n}-\tsigma{n-1}}{\dt}\right\|_\Hs^2 + \|h_n\|_\Hs^2\right)\Bigg),
    \end{aligned}\end{align*}
    and hence, by \eqref{stab_ineq3}, 
    \begin{align*}\begin{aligned}
        &\|v_m\|_\Hv^2 + \frac{1}{2}\|\sigmas{m}+\tsigma{m}\|_\Hs^2
        + \frac{1}{2}\|\sigma_m+\tsigma{m}\|_\Hs^2\\
        &+ \sum_{n=1}^m\left(\|v_n-v_{n-1}\|_\Hv^2 + \frac{1}{2}\|\sigma_{n}-\sigmas{n}\|_\Hs^2\right)
        + \nu\dt\sum_{n=1}^m\|\E(v_n)\|_\Hs^2\\
        \le& c_2\Bigg(\|v_0\|_\Hv^2 + \|\sigma_0\|_\Hs^2 + \|\tsigma{0}\|_\Hs^2\\
        &+ \dt\sum_{n=1}^N\left(\|f_n\|_{\Vv^*}^2 + \|p_n\|_\Hs^2 + \left\|\frac{\tsigma{n}-\tsigma{n-1}}{\dt}\right\|_\Hs^2 + \|h_n\|_\Hs^2\right)\Bigg),
    \end{aligned}\end{align*}
    where $c_2=ec_1$. 
    Since we have that for all $n=1,2,\ldots,N$,
    \begin{align*}\begin{aligned}
        \|f_n\|_{\Vv^*}^2
        &= \left\|\frac{1}{\dt}\int_{t_{n-1}}^{t_n}f(s)ds\right\|_{\Vv^*}^2
        \le \frac{1}{\dt}\|f\|_{L^2(t_{n-1},t_n;\Vv^*)}^2,\\
        \left\|\frac{\tsigma{n}-\tsigma{n-1}}{\dt}\right\|_\Hs^2
        &
        \le \frac{1}{\dt}\left\|\prd{p}{t}\right\|_{L^2(t_{n-1},t_n;\Hs)}^2,\quad
        \|h_n\|_{\Hs}^2
        \le \frac{1}{\dt}\|h\|_{L^2(t_{n-1},t_n;\Hs)}^2,\\
    \end{aligned}\end{align*}
    and 
    \[
        \|p_0\|_\Hs^2 + \dt\sum_{n=1}^N\|p_n\|_\Hs^2
        \le \|p\|_{C([0,T];\Hs)}^2\left(1 + \dt\sum_{n=1}^N 1\right)
        \le 2\|p\|_{C([0,T];\Hs)}^2,
    \]
    we obtain that for all $t\in[0,T]$,
    \begin{align*}\begin{aligned}
        &\|\bar{v}_\dt(t)\|_\Hv^2 
        + \frac{1}{2}\left\|\bar{\sigma}_\dt^*(t)+\bar{p}_\dt(t)\right\|_\Hs^2
        + \frac{1}{2}\left\|\bar{\sigma}_\dt(t)+\bar{p}_\dt(t)\right\|_\Hs^2\\
        &+ \int_0^t\left(\frac{1}{\dt}\|\hat{v}_\dt(s)-\bar{v}_\dt(s)\|_\Hv^2 
        + \frac{1}{2\dt}\left\|\bar{\sigma}_\dt(s) - \bar{\sigma}_\dt^*(s)\right\|_\Hs^2 + \nu\|\E(\bar{v}_\dt(s))\|_\Hs^2\right)ds\\
        \le& c_3(\|v_0\|_\Hv^2 + \|\sigma_0\|_\Hs^2 + \|f\|_{L^2(\Vv^*)}^2 + \|p\|_{H^1(\Hs)}^2 + \|h\|_{L^2(\Hs)}^2)
    \end{aligned}\end{align*}
    for a constant $c_3>0$, where we have used $C([0,T];\Hs)$ is continuously embedded to $H^1(\Hs)$.

    Furthermore, by the first equation of \eqref{proj} and the Korn inequality, it hold that 
    \begin{align*}
        \left\|\prd{\hat{v}_\dt}{t}\right\|_{L^2(\Vv^*)}^2
        &= \dt\sum_{n=1}^N \left(\sup_{0 \ne \varphi \in \Vv}\frac{1}{\|\varphi\|_\Vv}\left|\inner{\frac{v_n-v_{n-1}}{\dt}}{\varphi}{\Hv}\right|\right)^2 \\
        &\le 3c_K^2\nu^2\|\E(\bar{v}_\dt)\|_{L^2(\Hs)}^2 + 3c_K^2\|\bar{\sigma}_\dt^*\|_{L^2(\Hs)}^2 + 3\|\bar{f}_\dt\|_{L^2(\Vv^*)}^2,
    \end{align*}
    and hence, $(\|\prd{\hat{v}_\dt}{t}\|_{L^2(\Vv^*)})_{0<\dt<1}$ is bounded.

    (ii) By the second and third equations of \eqref{proj} and Proposition \ref{Prop:Phi} (ii) and (iv), we have for all $n=1, 2, \ldots, N$, 
    \begin{align*}
        &\quad~\left|(\sigma_n + p_n) - (\sigma_{n-1} + p_{n-1})\right| 
        = |\Pop{g_n}(\sigmas{n} + p_n) - \Pop{g_{n-1}}(\sigma_{n-1} + p_{n-1})| \\
        &\le |\Pop{g_n}(\sigmas{n} + p_n) - \Pop{g_{n-1}}(\sigmas{n} + p_n)| 
        + |\Pop{g_{n-1}}(\sigmas{n} + p_n) - \Pop{g_{n-1}}(\sigma_{n-1} + p_{n-1})|\\
        &\le |g_n - g_{n-1}| + |(\sigmas{n} + p_n) - (\sigma_{n-1} + p_{n-1})|\\
        %
        &\le \dt(|\E(v_n)| + |h_n|) + |p_n - p_{n-1}| + |g_n - g_{n-1}|
    \end{align*}
    on a.e. $\Omega$, and hence,
    \begin{align}\label{ineq:dsigma}\begin{aligned}        
        \|\sigma_n - \sigma_{n-1}\|_\Hs^2
        \le c_1\left(\dt^2(\|\E(v_n)\|_\Hs^2 + \|h_n\|_\Hs^2) 
        + \|p_n - p_{n-1}\|_\Hs^2 + \|g_n - g_{n-1}\|_\Lo^2\right),
    \end{aligned}\end{align}
    where $c_1 \defeq 7$.
    Since we have 
    \begin{align*}
        & \left\|\prd{\hat{\sigma}_\dt}{t}\right\|_{L^2(\Hs)}^2
        =\dt\sum_{n=1}^{N}\left\|\frac{\sigma_n-\sigma_{n-1}}{\dt}\right\|_\Hs^2\\
        \le& c_1\left(\|\E(\bar{v}_\dt)\|_{L^2(\Hs)}^2
        + \|\bar{h}_\dt\|_{L^2(\Hs)}^2
        + \left\|\prd{\hat{p}_\dt}{t}\right\|_{L^2(\Hs)}^2
        + \left\|\prd{\hat{g}_\dt}{t}\right\|_{L^2(\Lo)}^2\right),
    \end{align*}
    we obtain that the sequence $(\hat{\sigma}_\dt)$ is bounded in $H^1(\Hs)$.
\end{proof}



\subsection{Existence and uniqueness of solution to Problem \ref{Prob}}

\begin{proof}[Proof of Theorem \ref{Th:exact}]
    (Existence) According to Theorem \ref{Th:stab}, it can be shown that there exist a sequence $(\dt_k)_{k\in\N}$ and two functions 
    $v \in H^1(\Vv^*)\cap L^2(\Vv)$ (in particular, $v\in C([0,T];\Hv)$) and $\sigma \in H^1(\Lo)$ such that $\dt_k\rightarrow 0$ and for all $t\in[0,T]$
    \begin{align}\label{convvhwh}
        \hat{v}_{\dt_k}\rightharpoonup v
        &\quad\text{weakly in }H^1(\Vv^*),\\
        &\quad\begin{aligned}\label{convvhwl}
            \text{weakly star in }L^\infty([0,T];\Hv),
        \end{aligned}\\
        \hat{v}_{\dt_k}(t) \rightarrow v(t)
        &\quad\begin{aligned}\label{convvhwc}
            \text{strongly in }\Hv,
        \end{aligned}\\
        \bar{v}_{\dt_k}\rightharpoonup v
        &\quad\begin{aligned}\label{convvbwl}
            \text{weakly in }L^2(\Vv),
        \end{aligned}\\
        \hat{\sigma}_{\dt_k} \rightharpoonup \sigma
        &\quad\begin{aligned}\label{convshwh}
            \text{weakly in }H^1(\Hs),
        \end{aligned}\\
        \hat{\sigma}_{\dt_k}(t) \rightharpoonup \sigma(t)
        &\quad\begin{aligned}\label{convshwc}
            \text{weakly in }\Hs,
        \end{aligned}\\
        \bar{\sigma}^*_{\dt_k} \rightharpoonup \sigma
        &\quad\begin{aligned}\label{convsswl}
            \text{weakly star in }L^\infty(\Hs),
        \end{aligned}\\
        \bar{\sigma}_{\dt_k} \rightharpoonup \sigma
        &\quad\begin{aligned}\label{convsbwl}
            \text{weakly star in }L^\infty(\Hs),
        \end{aligned}
    \end{align}
    as $k\rightarrow\infty$. It should be noted that $(\hat{v}_{\dt_k})_{k\in\N}$ and $(\bar{v}_{\dt_k})_{k\in\N}$ possess a common limit function $v$ and $(\bar{\sigma}^*_{\dt_k})_{k\in\N}$ and $(\bar{\sigma}_{\dt_k})_{k\in\N}$ possess a common limit function $\sigma$. Indeed, the weak convergences \eqref{convvhwh}, \eqref{convvhwl}, \eqref{convvbwl}, \eqref{convshwh}, \eqref{convshwc}, \eqref{convsswl}, and \eqref{convsbwl} immediately follows by Theorem \ref{Th:stab} and Lemma \ref{Lem:AA}.
    If we define $\hat{v}_\dt^\circ \in C([0,T];\Vv)$ by 
    \[
        \hat{v}_\dt^\circ(t)\defeq\left\{\begin{aligned}
            &v_1 &&\mbox{ if }t\in[0,\dt],\\
            &\hat{v}_\dt(t) &&\mbox{ if }t\in[\dt,T],
        \end{aligned}\right.
    \]
    then we obtain that $(\hat{v}_\dt^\circ)_{0<\dt<1}$ is bounded in $H^1(\Vv^*)$ and $L^\infty(\Hv)$, and hence, 
    \[
        \hat{v}_{\dt_k}^\circ \rightarrow v 
        \quad\mbox{strongly in }C([0,T];\Hv)
    \]
    as $k\rightarrow\infty$, by the Aubin--Lions theorem \cite[Theorem II.5.16]{BF13}.
    For all $t\in(0,T]$, there exists $l\in\N$ such that $\dt_l \le t$. Since it holds that $\hat{v}_{\dt_k}^\circ(t)=\hat{v}_{\dt_k}(t)$ for all $k \ge l$, \eqref{convvhwc} holds.

    Since we have for all $t \in (t_{n-1},t_n), n=1,2,\ldots, N$, 
    \[\begin{aligned}
        \|\hat{\sigma}_\dt(t) - \bar{\sigma}_\dt(t)\|_{\Hs}
        &= \left\|\sigma_{n-1} + \frac{t-t_{n-1}}{\dt}(\sigma_n-\sigma_{n-1}) - \sigma_n\right\|_{\Hs}
        \le \dt\left\|\prd{\hat{\sigma}_\dt}{t}\right\|_{\Hs},
    \end{aligned}\]
    it holds that $\|\hat{\sigma}_\dt - \bar{\sigma}_\dt\|_{L^2(\Hs)}^2 \le \dt\|\hat{\sigma}_\dt\|_{H^1(\Hs)}^2$, and hence, the functions $\hat{\sigma}_\dt$ and $\bar{\sigma}_\dt$ possess a common limit function $\sigma$.

    Next, we demonstrate that the limit functions $(v, \sigma)$ satisfy \eqref{P} and $\sigma(t) \in K(t)$ for all $t \in [0, T]$.

    By the third equation of \eqref{proj} with $\dt\defeq\dt_k$, it holds that for all $t\in[t_{n-1},t_n],n=1,2,\ldots,N$,
    \[\begin{aligned}
        |(\hat{\sigma}_{\dt_k}(t) + \hat{p}_{\dt_k}(t))^D| 
        &= \left|\frac{t-t_{n-1}}{\dt}(\sigma_n+p_n)^D 
        + \frac{t_n-t}{\dt}(\sigma_{n-1} + p_{n-1})^D\right|\\
        &\le \frac{t-t_{n-1}}{\dt}g_n + \frac{t_n-t}{\dt}g_{n-1}
        = \hat{g}_{\dt_k}(t)
    \end{aligned}\]
    a.e. on $\Omega$, i.e. $\hat{\sigma}_{\dt_k}(t) + \hat{p}_{\dt_k}(t) \in \left\{\tau\in\Hs : |\tau^D|\le \hat{g}_{\dt_k}(t) \mbox{ a.e. in }\Omega\right\}$ for all $t\in[0,T]$. 
    By $g\in H^1(\Lo)$, we have that $\hat{g}_\dt \rightarrow g$ strongly in $C([0,T];\Lo)$. By the Riesz--Fischer theorem, $\hat{g}_\dt(t) \rightarrow g(t)$ a.e. on $\Omega$ for all $t\in[0,T]$. For fixed $k\in\N$ and $t\in[0,T]$, since 
    \[
        \left\{\tau\in\Hs : |\tau^D|\le \sup_{l \ge k}\hat{g}_{\dt_l}(t) \mbox{ a.e. in }\Omega\right\}    
    \] 
    is closed and convex in the strong topology of $\Hs$, it is also closed and convex in the weak topology of $\Hs$. By \eqref{convshwc}, we obtain for all $k\in\N$,
    \[
        \sigma(t)+p(t) \in \left\{\tau\in\Hs : |\tau^D|\le \sup_{l \ge k}\hat{g}_{\dt_l}(t) \mbox{ a.e. in }\Omega\right\},
    \]
    and hence, 
    \[\begin{aligned}
        \sigma(t)+p(t) &\in \bigcap_{k\in\N}\left\{\tau\in\Hs : |\tau^D|\le \sup_{l \ge k}\hat{g}_{\dt_l}(t) \mbox{ a.e. in }\Omega\right\}\\
        &= \left\{\tau\in\Hs : |\tau^D|\le \inf_{k\in\N}\sup_{l \ge k}\hat{g}_{\dt_l}(t) \mbox{ a.e. in }\Omega\right\}\\
        &= \left\{\tau\in\Hs : |\tau^D|\le g(t) \mbox{ a.e. in }\Omega\right\}
        = \tilde{K}(t),
    \end{aligned}\]
    i.e. $\sigma(t)\in K(t)$ for all $t\in[0,T]$.

    By the first equation of \eqref{proj} with $\dt\defeq\dt_k$, it holds that for all $\varphi\in\Vv$ and $\theta\in C^\infty_0(0,T)$,
    \begin{align*}
        &\int_0^T\left(\blaket{\prd{\hat{v}_{\dt_k}}{t}}{\theta\varphi}
        + \nu(\E(\bar{v}_{\dt_k}), \E(\theta\varphi))_\Hs + (\bar{\sigma}^*_{\dt_k}, \E(\theta\varphi))_\Hs \right)dt\\
        =& \int_0^T\blaket{\bar{f}_{\dt_k}}{\theta\varphi} dt.
    \end{align*}
    By taking $k\rightarrow\infty$, we obtain that for all $\varphi\in\Vv$ and $\theta\in C^\infty_0(0,T)$,
    \begin{align*}
        \int_0^T\left(\blaket{\prd{v}{t}}{\theta\varphi}
        + \nu(\E(v), \E(\theta\varphi))_\Hs + (\sigma, \E(\theta\varphi))_\Hs \right)dt
        = \int_0^T\blaket{f}{\theta\varphi} dt,
    \end{align*}
    which implies the first equation of \eqref{P}. 
    
    By the second and third equations of \eqref{proj} and Proposition \ref{Prop:Phi} (iv), it holds that for all $t \in (t_{n-1},t_n), n=1,2,\ldots,N$ and $\tau \in C([0,T];\Hs)$ with $\tau(t) \in K(t)$,
    \begin{align*}
        &\left(\prd{\hat{\sigma}_{\dt_k}}{t}(t) - \E(\bar{v}_{\dt_k}(t)) - \bar{h}_{\dt_k}(t), 
        \bar{\sigma}_{\dt_k}(t) - \bar{\tau}_{\dt_k}(t)\right)_\Hs\\
        =& \left(\frac{\sigma_n - \sigma_{n-1}}{\dt} - \E(v_n) - h_n, 
        \sigma_n - \tau_n\right)_\Hs
        = \frac{1}{\dt}(\sigma_n-\sigmas{n}, \sigma_n - \tau_n)_\Hs \\
        =& \frac{1}{\dt}((\sigma_n+p_n)-(\sigmas{n}+p_n), 
        (\sigma_n+p_n) - (\tau_n+p_n))_\Hs\le 0,
    \end{align*}
    and hence, it holds that for all $t\in[0,T]$,
    \begin{align*}
        \int_0^t \left(\prd{\hat{\sigma}_{\dt_k}}{t} - \E(\bar{v}_{\dt_k}) - \bar{h}_{\dt_k}, 
        \bar{\sigma}_{\dt_k} - \bar{\tau}_{\dt_k}\right)_\Hs ds
        \le 0.
    \end{align*}
    Since $\bar{\tau}_\dt$ and $\bar{h}_\dt$ converge to $\tau$ and $h$, respectively, strongly in $L^2(\Hs)$, we have that for all $t\in[0,T]$,
    \[\begin{aligned}
        \int_0^t \left(\prd{\hat{\sigma}_{\dt_k}}{t} - \E(\bar{v}_{\dt_k}) - \bar{h}_{\dt_k}, \bar{\tau}_{\dt_k}\right)_\Hs ds
        &\rightarrow \int_0^t \left(\prd{\sigma}{t} - \E(v) - h,\tau\right)_\Hs ds,\\
        \int_0^t \left(\bar{h}_{\dt_k}, 
        \bar{\sigma}_{\dt_k}\right)_\Hs ds
        &\rightarrow \int_0^t (h, \sigma)_\Hs ds
    \end{aligned}\] 
    as $k \rightarrow \infty$. 
    Furthermore, by Theorem \ref{Th:stab}, \eqref{convshwc}, and \eqref{convvhwc}, we obtain that
    \[\begin{aligned}
        &\liminf_{k\rightarrow\infty}
        \int_0^t \left(\prd{\hat{\sigma}_{\dt_k}}{t}, 
        \bar{\sigma}_{\dt_k}\right)_\Hs ds\\
        =&\frac{1}{2}\liminf_{k\rightarrow\infty}
        (\|\hat{\sigma}_{\dt_k}(t)\|_\Hs^2 - \|\sigma_0\|_\Hs^2)
        +\liminf_{k\rightarrow\infty}
        \int_0^t \left(\prd{\hat{\sigma}_{\dt_k}}{t}, 
        \bar{\sigma}_{\dt_k}-\hat{\sigma}_{\dt_k}\right)_\Hs ds\\
        \ge& \frac{1}{2}\|\sigma(t)\|_\Hs^2ds 
        - \frac{1}{2}\|\sigma_0\|_\Hs^2
        = \int_0^t \left(\prd{\sigma}{t}, \sigma\right)_\Hs ds,\\
        &\liminf_{k\rightarrow\infty}
        \left(-\int_0^t \left(\E(\bar{v}_{\dt_k}), 
        \bar{\sigma}_{\dt_k}\right)_\Hs ds\right)\\
        =& \liminf_{k\rightarrow\infty}
        \left(-\int_0^t \left(\E(\bar{v}_{\dt_k}), 
        \bar{\sigma}^*_{\dt_k}\right)_\Hs ds
        + \int_0^t \left(\E(\bar{v}_{\dt_k}), 
        \bar{\sigma}^*_{\dt_k} - \bar{\sigma}_{\dt_k}\right)_\Hs ds\right)
        \\
        =&\liminf_{k\rightarrow\infty}
        \int_0^t\left(\blaket{\prd{\hat{v}_{\dt_k}}{t}}{\bar{v}_{\dt_k}}
        + \nu\|\E(\bar{v}_{\dt_k})\|_\Hs^2
        - \blaket{\bar{f}_{\dt_k}}{\bar{v}_{\dt_k}} \right)ds\\
        %
        %
        \ge& \frac{1}{2}\|v(t)\|_\Hs^2 
        - \frac{1}{2}\|v(0)\|_\Hs^2
        + \int_0^t\left(\nu\|\E(v)\|_\Hs^2
        - \blaket{f}{v} \right)dt
        = -\int_0^t \left(\E(v), \sigma\right)_\Hs ds,
    \end{aligned}\]
    where we have used the first equation of \eqref{P}. 
    Therefore, we obtain that for all $t\in[0,T]$ and $\tau \in C([0,T];\Hs)$ with $\tau(t) \in K(t)$,
    \begin{align*}
        &\int_0^t\left(\prd{\sigma}{t} - \E(v) - h,\sigma - \tau\right)_\Hs ds\\
        \le& \liminf_{k\rightarrow\infty}\int_0^t \left(\prd{\hat{\sigma}_{\dt_k}}{t} - \E(\bar{v}_{\dt_k}) - \bar{h}_{\dt_k}, 
        \bar{\sigma}_{\dt_k} - \bar{\tau}_{\dt_k}\right)_\Hs ds 
        \le 0,
    \end{align*}
    which implies the second inequality of \eqref{P}.

    (Uniqueness) Let $(v_1, \sigma_1)$ and $(v_2, \sigma_2)$ be the solution to Problem \ref{Prob}. If we set $e_v \defeq v_1-v_2$ and $e_\sigma \defeq \sigma_1-\sigma_2$, by the first equation of \eqref{P}, we have for all $\varphi \in\Vv$,
    \[
        \blaket{\prd{e_v}{t}}{\varphi}
        + \nu (\E(e_v),\E(\varphi))_\Hs + (e_\sigma,\E(\varphi))_\Hs
        = 0.
    \]
    Putting $\varphi\defeq e_v$ and integrating over time, we obtain that for all $t\in[0,T]$,
    \begin{align}\label{ineq:uniq1}
        \frac{1}{2}\|e_v(t)\|_\Hv^2
        + \int_0^t\left(\nu \|\E(e_v)\|_\Hs^2 + (e_\sigma,\E(e_v))_\Hs\right)ds
        = 0,
    \end{align}
    where we have used $e_v(0) = v_1(0) - v_2(0) = 0$.
    
    Since $(v_1, \sigma_1)$ is the solution to Problem \ref{Prob}, by putting $\tau\defeq\sigma_2$ in \eqref{P}, we obtain that 
    \begin{align}\label{ineq:uniq2}
        \left(\prd{\sigma_1}{t}-\E(v_1),\sigma_1-\sigma_2\right)_\Hs
        \le \left(h,\sigma_1-\sigma_2\right)_\Hs.
    \end{align}
    On the other hand, since $(v_2, \sigma_2)$ is the solution to Problem \ref{Prob}, by putting $\tau\defeq\sigma_1$ in \eqref{P}, we obtain 
    \begin{align}\label{ineq:uniq3}
        \left(\prd{\sigma_2}{t}-\E(v_2),\sigma_2-\sigma_1\right)_\Hs
        \le \left(h,\sigma_2-\sigma_1\right)_\Hs.
    \end{align}
    Adding \eqref{ineq:uniq2} and \eqref{ineq:uniq3} together, we get
    \begin{align*}
        \left(\prd{e_\sigma}{t}-\E(e_v),e_\sigma\right)_\Hs \le 0.
    \end{align*}
    Integrating for time, we obtain that for all $t\in[0,T]$,
    \begin{align}\label{ineq:uniq4}
        \frac{1}{2}\|e_\sigma(t)\|_\Hs^2
        - \int_0^t\left(\E(e_v),e_\sigma\right)_\Hs ds \le 0,
    \end{align}
    where we have used $e_\sigma(0) = \sigma_1(0) - \sigma_2(0) = 0$. Thus, summing up \eqref{ineq:uniq1} and \eqref{ineq:uniq4}, it holds that for all $t\in[0,T]$, 
    \[
        \frac{1}{2}\|e_v(t)\|_\Hv^2 
        + \frac{1}{2}\|e_\sigma(t)\|_\Hs^2
        + \nu\int_0^t\|\E(e_v)\|_\Hs^2ds \le 0,
    \]
    and hence, $v_1-v_2 = e_v = 0$ and $\sigma_1-\sigma_2 = e_\sigma = 0$ on $[0,T]$.
\end{proof}

\subsection{Strong convergence}\label{sec:conv}

In this subsection, we prove Theorem \ref{Th:strong}. We first state the following lemma:

\begin{Lem}\label{Lem:stab2}
    Under the assumption of Theorem \ref{Th:exact}, if $v_0\in\Vv$ and $f\in L^2(\Hs)$, then there exists a constant $c > 0$ independent of $\dt$ such that
    \[\begin{aligned}
        &\left\|\prd{\hat{v}_\dt}{t}\right\|_{L^2(\Hv)}
        + \|\bar{v}_\dt\|_{L^\infty(\Vv)} 
        + \frac{1}{\dt}\|\bar{v}_\dt - \bar{v}_\dt^*\|_{L^2(\Vv)}\\
        \le~& c(\|v_0\|_\Vv + \|\sigma_0\|_\Hs + \|f\|_{L^2(\Hv)} + \|p\|_{H^1(\Hs)} + \|h\|_{L^2(\Hs)} + \|g\|_{H^1(\Lo)}).
    \end{aligned}\] 
    In particular, $(\|\hat{v}_\dt\|_{H^1(\Hv)})_{0<\dt<1}$ and $(\|\bar{v}_\dt\|_{L^\infty(\Vv)})_{0<\dt<1}$ are bounded.
\end{Lem}

\begin{proof}
    Putting $\varphi \defeq v_n - v_{n-1}$ in the first equation of \eqref{proj}, we obtain 
    \begin{align*}
        &\frac{\dt}{2}\|\frac{v_n-v_{n-1}}{2}\|_\Hv^2
        + \nu (\E(v_n),\E(v_n-v_{n-1}))_\Hs 
        + (\sigmas{n},\E(v_n-v_{n-1}))_\Hs\\
        =& (f_n, v_n-v_{n-1})_\Hv.
    \end{align*}
    Applying the second equation of \eqref{proj} yields
    \begin{align*}
        (\sigmas{n},\E(v_n-v_{n-1}))_\Hs
        =& \dt(\E(v_n), \E(v_n-v_{n-1}))_\Hs
        + (\sigma_n, \E(v_n))_\Hs - (\sigma_{n-1}, \E(v_{n-1}))_\Hs\\
        & - (\sigma_n - \sigma_{n-1}, \E(v_n))_\Hs + \dt(h_n, \E(v_n-v_{n-1}))_\Hs,
    \end{align*}
    and hence,
    \begin{align*}
        &\frac{\nu+\dt}{2}(\|\E(v_n)\|_\Hs^2 - \|\E(v_{n-1})\|_\Hs^2 + \|\E(v_n-v_{n-1})\|_\Hs^2)\\
        &+ (\sigma_n, \E(v_n))_\Hs - (\sigma_{n-1}, \E(v_{n-1}))_\Hs + \dt\left\|\frac{v_n-v_{n-1}}{\dt}\right\|_\Hv^2\\
        =& (f_n, v_n-v_{n-1})_\Hv 
        + (\sigma_n - \sigma_{n-1}, \E(v_n))_\Hs 
        - \dt(h_n, \E(v_n-v_{n-1}))_\Hs\\
        %
        %
        \le& \frac{\dt}{2}\|f_n\|_{\Hv}^2
        + \frac{\dt}{2}\left\|\frac{v_n-v_{n-1}}{\dt}\right\|_\Hv^2
        + \frac{\dt}{2}\left\|\frac{\sigma_n - \sigma_{n-1}}{\dt}\right\|_\Hs^2
        + \frac{\dt}{2}\|\E(v_n)\|_\Hs^2\\
        &+ \frac{\dt}{2}\|h_n\|_\Hs^2
        + \frac{\dt}{2}\|\E(v_n-v_{n-1})\|_\Hs^2,
    \end{align*}
    which implies that 
    \begin{align*}
        &\nu(\|\E(v_n)\|_\Hs^2 - \|\E(v_{n-1})\|_\Hs^2 + \|\E(v_n-v_{n-1})\|_\Hs^2)\\
        &+ 2(\sigma_n, \E(v_n))_\Hs - 2(\sigma_{n-1}, \E(v_{n-1}))_\Hs + \dt\left\|\frac{v_n-v_{n-1}}{\dt}\right\|_\Hv^2\\
        \le& \dt\|\E(v_{n-1})\|_\Hs^2
        + \dt\|f_n\|_{\Hv}^2
        + \dt\left\|\frac{\sigma_n - \sigma_{n-1}}{\dt}\right\|_\Hs^2
        + \dt\|h_n\|_\Hs^2.
    \end{align*}
    By summing up for $n = 1,2,\ldots,m$, where $m\le N$, it holds that
    \begin{align*}
        & \nu\|\E(v_m)\|_\Hs^2 + 2(\sigma_m, \E(v_m))_\Hs\\
        &+ \dt\sum_{n=1}^m\left(
            \frac{\nu}{\dt}\|\E(v_n-v_{n-1})\|_\Hs^2 + \left\|\frac{v_n-v_{n-1}}{\dt}\right\|_\Hv^2\right)\\
        \le& \nu\|\E(v_0)\|_\Vv^2 + 2(\sigma_0, \E(v_0))_\Hs\\
        &+ \dt\sum_{n=1}^m\left(
            \|\E(v_{n-1})\|_\Hs^2 + \|f_n\|_{\Hv}^2
            + \left\|\frac{\sigma_n - \sigma_{n-1}}{\dt}\right\|_\Hs^2
            + \|h_n\|_\Hs^2\right),
    \end{align*}
    which implies that 
    \begin{align*}
        &\nu\|\E(\bar{v}_\dt(t))\|_\Hs^2 
        + 2(\bar{\sigma}_\dt(t), \E(\bar{v}_\dt(t)))_\Hs\\
        &+ \int_0^t\left(
            \frac{\nu}{\dt}\|\E(\hat{v}_\dt(s)-\bar{v}_\dt(s))\|_\Hs^2 + \left\|\prd{\hat{v}_\dt}{t}\right\|_\Hs^2\right)ds\\
        \le& (\nu+\dt)\|\E(v_0)\|_\Hs^2 + 2(\sigma_0, \E(v_0))_\Hs\\
        &+ \|\E(\bar{v}_\dt)\|_{L^2(\Hs)}^2 
        + \|\bar{f}_\dt\|_{L^2(\Hv)}^2
        + \left\|\hat{\sigma}_{\dt}'\right\|_{L^2(\Hs)}^2
        + \|\bar{h}_{\dt}\|_{L^2(\Hs)}^2,
    \end{align*}
    for all $t\in[0,T]$. Since we have that
    \begin{align*}
        |2(\bar{\sigma}_\dt(t), \E(\bar{v}_\dt(t)))|
        &
        \le \frac{\nu}{2}\|\E(\bar{v}_\dt(t))\|_\Hs^2
        + \frac{2}{\nu}\|\bar{\sigma}_\dt\|_{L^\infty(\Hs)}^2,\\
        |2(\sigma_0, \E(v_0))|
        &\le \|\sigma_0\|_\Hs^2 + \|\E(v_0)\|_\Hs^2,
    \end{align*}
    we obtain 
    \begin{align*}
        &\frac{\nu}{2}\|\E(\bar{v}_\dt(t))\|_\Hs^2 
        + \int_0^t\left(
            \frac{\nu}{\dt}\|\E(\hat{v}_\dt(s)-\bar{v}_\dt(s))\|_\Hs^2 + \left\|\prd{\hat{v}_\dt}{t}\right\|_\Hs^2
            \right)ds\\
        \le& (\nu+\dt+1)\|\E(v_0)\|_\Hs^2
        + \|\sigma_0\|_\Hs^2
        + \frac{2}{\nu}\|\bar{\sigma}_\dt\|_{L^\infty(\Hs)}^2
        + \|\E(\bar{v}_\dt)\|_{L^2(\Hs)}^2 \\
        &+ \|\bar{f}_\dt\|_{L^2(\Hv)}^2 
        + \left\|\prd{\hat{\sigma}_{\dt}}{t}\right\|_{L^2(\Hs)}^2
        + \|\bar{h}_{\dt}\|_{L^2(\Hs)}^2.
    \end{align*}
    Therefore, applying Theorem \ref{Th:stab} and the Korn inequality leads us to the conclusion.
\end{proof}

\begin{proof}[Proof of Theorem \ref{Th:strong}]
    From the proof of Theorem \ref{Th:exact}, it holds that for all $\varphi\in\Vv$ and $\tau \in C([0,T];\Hs)$ with $\tau(t) \in K(t)$ for all $t\in[0,T]$,
    \begin{align}\label{approaxP}
        \left\{\begin{aligned}
            &\left(\prd{\hat{v}_\dt}{t},\varphi\right)_\Hv
            + \nu(\E(\bar{v}_\dt), \E(\varphi))_\Hs 
            + (\bar{\sigma}_\dt, \E(\varphi))_\Hs
            = \left(\bar{f}_\dt,\varphi\right)_\Hv,\\
            &\left(\prd{\hat{\sigma}_\dt}{t} - \E(\bar{v}_\dt), 
            \bar{\sigma}_\dt - \tau_\dt\right)_\Hs
            \le \left(\bar{h}_\dt, 
            \bar{\sigma}_\dt - \tau_\dt\right)_\Hs,
        \end{aligned}\right.
    \end{align}
    a.e. on $(0,T)$, where $\tau_\dt(t)\defeq\tau(t_n)$ for $t\in(t_{n-1},t_n]$ and $n=1,2,\ldots,N$.
    We define the error terms as $\hat{e}_\dt \defeq \hat{v}_\dt - v$, $\bar{e}_\dt \defeq \bar{v}_\dt - v$, $\bar{\eps}_\dt \defeq \bar{\sigma}_\dt - \sigma$, $\bar{\eps}^*_\dt \defeq \bar{\sigma}^*_\dt - \sigma$, and $\hat{\eps}_\dt \defeq \hat{\sigma}_\dt - \sigma$. Using the first equations of \eqref{P} and \eqref{approaxP}, we get for all $\varphi\in\Vv$,
    \[
        \left(\prd{\hat{e}_\dt}{t},\varphi\right)_\Hv
        + \nu(\E(\bar{e}_\dt), \E(\varphi))_\Hs 
        + (\bar{\eps}^*_\dt, \E(\varphi))_\Hs
        = (\bar{f}_\dt-f,\varphi)_\Hv
    \]
    a.e. on $(0,T)$. Putting $\varphi\defeq\bar{e}_\dt$, we obtain
    \begin{align*}
        \left(\prd{\hat{e}_\dt}{t},\bar{e}_\dt\right)_\Hv
        + \nu\|\E(\bar{e}_\dt)\|_\Hs^2
        + (\bar{\eps}^*_\dt, \E(\bar{e}_\dt))_\Hs
        = (\bar{f}_\dt-f,\bar{e}_\dt)_\Hv
    \end{align*}
    a.e. on $(0,T)$, which implies that 
    \begin{align}\label{ineq:err1}\begin{aligned}
        &\left(\prd{\hat{e}_\dt}{t},\hat{e}_\dt\right)_\Hv
        + \nu\|\E(\bar{e}_\dt)\|_\Hs^2
        + (\bar{\eps}^*_\dt, \E(\bar{e}_\dt))_\Hs\\
        \le& \|\bar{f}_\dt-f\|_\Hv\|\bar{e}_\dt\|_\Hv
        + \left(\prd{\hat{e}_\dt}{t},\hat{v}_\dt - \bar{v}_\dt\right)_\Hv.
    \end{aligned}\end{align}
    Applying $\tau\defeq\sigma_\dt$ in the second inequality of \eqref{approaxP}, where $\sigma_\dt(t)\defeq\sigma(t_n)$ for $t\in(t_{n-1},t_n]$ and $n=1,2,\ldots,N$, we have that
    \begin{align*}
        \left(\prd{\hat{\sigma}_\dt}{t} - \E(\bar{v}_\dt)_\Hs, 
        \bar{\sigma}_\dt - \sigma_\dt\right)_\Hs
        \le \left(\bar{h}_\dt, \bar{\sigma}_\dt - \sigma_\dt\right)_\Hs
    \end{align*}
    a.e. on $(0,T)$, which implies that 
    \begin{align}\label{ineq:err2-1}\begin{aligned}
        &\quad\left(\prd{\hat{\sigma}_\dt}{t}, \hat{\eps}_\dt\right)_\Hs
        - \left(\E(\bar{v}_\dt), \bar{\eps}^*_\dt\right)_\Hs\\
        %
        %
        &\le \left(\bar{h}_\dt - h, \bar{\sigma}_\dt - \sigma_\dt\right)_\Hs
        + \left(h, \bar{\sigma}_\dt - \sigma_\dt\right)_\Hs
        + \left(\prd{\hat{\eps}_\dt}{t}, \hat{\sigma}_\dt - \bar{\sigma}_\dt\right)_\Hs\\
        &\quad+ \left(\prd{\sigma}{t}, 
        \hat{\sigma}_\dt - \bar{\sigma}_\dt\right)_\Hs
        - \left(\E(\bar{e}_\dt), 
        \bar{\sigma}^*_\dt - \bar{\sigma}_\dt\right)_\Hs
        - \left(\E(v), \bar{\sigma}^*_\dt - \bar{\sigma}_\dt\right)_\Hs\\
        &\quad- \left(\prd{\hat{\sigma}_\dt}{t} - \E(\bar{v}_\dt), 
        \sigma - \sigma_\dt\right)_\Hs.
    \end{aligned}\end{align}
    Similarly, using $\tau\defeq\Pop{g}(\bar{\sigma}_\dt+p)-p$ in the second inequality of \eqref{P}, we find
    \begin{align*}
        \left(\frac{d\sigma}{dt}-\E(v),\sigma+p
        - \Pop{g}(\bar{\sigma}_\dt+p)\right)_\Hs
        \le \left(h,\sigma+p-\Pop{g}(\bar{\sigma}_\dt+p)\right)_\Hs
    \end{align*}
    a.e. on $(0,T)$, which implies that
    \begin{align}\label{ineq:err2-2}\begin{aligned}
        &\quad-\left(\prd{\sigma}{t}, 
        \hat{\eps}_\dt\right)_\Hs
        + \left(\E(v), \bar{\eps}^*_\dt\right)_\Hs
        =\left(\prd{\sigma}{t}, 
        \sigma - \hat{\sigma}_\dt\right)_\Hs
        - \left(\E(v), 
        \bar{\sigma}^*_\dt - \sigma\right)_\Hs\\
        &\le \left(h,
        \sigma+p-\Pop{g}(\bar{\sigma}_\dt+p)\right)_\Hs
        + \left(\prd{\sigma}{t}, 
        \Pop{g}(\bar{\sigma}_\dt+p) 
        - (\hat{\sigma}_\dt+p)\right)_\Hs\\
        &\quad- \left(\E(v), 
        \Pop{g}(\bar{\sigma}_\dt+p) 
        - (\bar{\sigma}^*_\dt+p)\right)_\Hs.
    \end{aligned}\end{align}
    Adding \eqref{ineq:err2-1} and \eqref{ineq:err2-2}, we deduce
    \begin{align}\label{ineq:err2}\begin{aligned}
        &\quad\left(\prd{\hat{\eps}_\dt}{t}, \hat{\eps}_\dt\right)_\Hs
        - \left(\E(\bar{e}_\dt), \bar{\eps}^*_\dt\right)_\Hs\\
        &\le  \left(\bar{h}_\dt - h, \bar{\sigma}_\dt - \sigma_\dt\right)_\Hs
        + \left(h, 
        ((\bar{\sigma}_\dt + p) 
        - \Pop{g}(\bar{\sigma}_\dt+p))
        - (\sigma_\dt - \sigma)\right)_\Hs\\
        &\quad+ \left(\prd{\hat{\eps}_\dt}{t}, 
        \hat{\sigma}_\dt - \bar{\sigma}_\dt\right)_\Hs
        + \left(\prd{\sigma}{t}, 
        \Pop{g}(\bar{\sigma}_\dt+p) 
        - (\bar{\sigma}_\dt+p)\right)_\Hs
        - \left(\E(\bar{e}_\dt), 
        \bar{\sigma}^*_\dt - \bar{\sigma}_\dt\right)_\Hs\\
        &\quad- \left(\E(v), 
        \Pop{g}(\bar{\sigma}_\dt+p) - (\bar{\sigma}_\dt+p)\right)_\Hs
        - \left(\prd{\hat{\sigma}_\dt}{t} - \E(\bar{v}_\dt), 
        \sigma - \sigma_\dt\right)_\Hs\\
        &\le \|\bar{h}_\dt - h\|_\Hs
        \|\bar{\sigma}_\dt - \sigma_\dt\|_\Hs
        + \left\|\prd{\hat{\eps}_\dt}{t}\right\|_\Hs
        \|\hat{\sigma}_\dt - \bar{\sigma}_\dt\|_\Hs\\
        &\quad+ \|\E(\bar{e}_\dt)\|_\Hs
        \|\bar{\sigma}^*_\dt - \bar{\sigma}_\dt\|_\Hs
        + \left(\|h\|_\Hs
        + \left\|\prd{\sigma}{t}\right\|_\Hs + \|\E(v)\|_\Hs\right) 
        \|g - g_\dt\|_\Lo\\
        &\quad+ \left(\|h\|_\Hs
        + \left\|\prd{\hat{\sigma}_\dt}{t} - \E(\bar{v}_\dt)\right\|_\Hs\right) \|\sigma_\dt - \sigma\|_\Hs
    \end{aligned}\end{align}
    a.e. on $(0,T)$, where we have used $\bar{\sigma}_\dt+p = \Pop{g_\dt}(\bar{\sigma}_\dt+p)$ and Proposition \ref{Prop:Phi} (ii).
    Additionally, by adding \eqref{ineq:err1} and \eqref{ineq:err2}, we get
    \[\begin{aligned}
        &\quad \left(\prd{\hat{e}_\dt}{t},\hat{e}_\dt\right)_\Hv
        + \nu\|\E(\bar{e}_\dt)\|_\Hs^2
        + \left(\prd{\hat{\eps}_\dt}{t}, 
        \hat{\eps}_\dt\right)_\Hs\\
        &\le \|\bar{f}_\dt-f\|_{\Hv}\|\bar{e}_\dt\|_\Hv
        + \left(\prd{\hat{e}_\dt}{t},\hat{v}_\dt - \bar{v}_\dt\right)_\Hv\\
        &\quad+  \|\bar{h}_\dt - h\|_\Hs
        \|\bar{\sigma}_\dt - \sigma_\dt\|_\Hs
        + \left\|\prd{\hat{\eps}_\dt}{t}\right\|_\Hs
        \|\hat{\sigma}_\dt - \bar{\sigma}_\dt\|_\Hs\\
        &\quad+ \|\E(\bar{e}_\dt)\|_\Hs
        \|\bar{\sigma}^*_\dt - \bar{\sigma}_\dt\|_\Hs
        + \left(\|h\|_\Hs
        + \left\|\prd{\sigma}{t}\right\|_\Hs + \|\E(v)\|_\Hs\right) 
        \|g - g_\dt\|_\Lo\\
        &\quad+ \left(\|h\|_\Hs
        + \left\|\prd{\hat{\sigma}_\dt}{t} - \E(\bar{v}_\dt)\right\|_\Hs\right) \|\sigma_\dt - \sigma\|_\Hs
    \end{aligned}\]
    a.e. on $(0,T)$. 
    Hence, by integrating it, we have that for all $t\in[0,T]$,
    \[\begin{aligned}
        &\quad \|\hat{e}_\dt(t)\|_\Hv^2
        + \nu\|\E(\bar{e}_\dt)\|_{L^2(0,t;\Hs)}^2
        + \|\hat{\eps}_\dt(t)\|_\Hs^2\\
        &\le \|\bar{f}_\dt-f\|_{L^2(\Hv)}\|\bar{e}_\dt\|_{L^2(\Hv)}
        + \left\|\prd{\hat{e}_\dt}{t}\right\|_{L^2(\Hv)}
        \|\hat{v}_\dt - \bar{v}_\dt\|_{L^2(\Hv)}\\
        &\quad+  \|\bar{h}_\dt - h\|_{L^2(\Hs)}
        \|\bar{\sigma}_\dt - \sigma_\dt\|_{L^2(\Hs)}
        + \left\|\prd{\hat{\eps}_\dt}{t}\right\|_{L^2(\Hs)}
        \|\hat{\sigma}_\dt - \bar{\sigma}_\dt\|_{L^2(\Hs)}\\
        &\quad+ \|\E(\bar{e}_\dt)\|_{L^2(\Hs)}
        \|\bar{\sigma}^*_\dt - \bar{\sigma}_\dt\|_{L^2(\Hs)}\\
        &\quad+ \left(\|h\|_{L^2(\Hs)}
        + \left\|\prd{\sigma}{t}\right\|_{L^2(\Hs)} + \|\E(v)\|_{L^2(\Hs)}\right) 
        \|g - g_\dt\|_{L^2(\Lo)}\\
        &\quad+ \left(\|h\|_{L^2(\Hs)}
        + \left\|\prd{\hat{\sigma}_\dt}{t} - \E(\bar{v}_\dt)\right\|_{L^2(\Hs)}\right) \|\sigma_\dt - \sigma\|_{L^2(\Hs)}.
    \end{aligned}\]
    By Theorem \ref{Th:stab}, the Korn inequality, and Lemma \ref{Lem:stab2}, there exists a constant $c>0$ such that for all $0<\dt<1$,
    \[\begin{aligned}
        &\quad \|\hat{e}_\dt\|_{L^\infty(\Hv)}^2
        + \|\bar{e}_\dt\|_{L^2(\Vv)}^2
        + \|\hat{\eps}_\dt(t)\|_{L^\infty(\Hs)}^2\\
        &\le c(\|\bar{f}_\dt-f\|_{L^2(\Hv)}
        + \|\hat{v}_\dt - \bar{v}_\dt\|_{L^2(\Hv)}
        +  \|\bar{h}_\dt - h\|_{L^2(\Hs)}
        + \|\hat{\sigma}_\dt - \bar{\sigma}_\dt\|_{L^2(\Hs)}\\
        &\quad
        + \|\bar{\sigma}^*_\dt - \bar{\sigma}_\dt\|_{L^2(\Hs)}
        + \|g - g_\dt\|_{L^2(\Lo)}
        + \|\sigma_\dt - \sigma\|_{L^2(\Hs)}),
    \end{aligned}\]
    where we have used that 
    $
        \left\|\prd{\hat{e}_\dt}{t}\right\|_{L^2(\Hs)}
        \le \left\|\prd{\hat{v}_\dt}{t}\right\|_{L^2(\Hs)}
        + \left\|\prd{v}{t}\right\|_{L^2(\Hs)}
    $
    and\\ $(\left\|\prd{\hat{e}_\dt}{t}\right\|_{L^2(\Hs)})_{0<\dt<1}$ is bounded.
    Since 
    \begin{align*}
        \bar{f}_\dt &\rightarrow f\quad\mbox{strongly in }L^2(\Hv),\\
        \bar{h}_\dt &\rightarrow h\quad\mbox{strongly in }L^2(\Hs),\\
        g_\dt &\rightarrow g\quad\mbox{strongly in }L^2(\Lo),\\
        \sigma_\dt &\rightarrow \sigma\quad\mbox{strongly in }L^2(\Hs),\\
        \hat{v}_\dt - \bar{v}_\dt &\rightarrow 0\quad\mbox{strongly in }L^2(\Hv),\\
        \hat{\sigma}_\dt - \bar{\sigma}_\dt &\rightarrow 0\quad\mbox{strongly in }L^2(\Hs),\\
        \bar{\sigma}^*_\dt - \bar{\sigma}_\dt &\rightarrow 0\quad\mbox{strongly in }L^2(\Hs),
    \end{align*}
    as $\dt\rightarrow 0$, we obtain the conclusion.
\end{proof}

\section{Conclusion}\label{sec:Conclusion}

For Problem \ref{Prob} of a perfect plasticity model with a time-dependent yield surface, we proposed a new numerical scheme \eqref{proj} and proved its stability in Theorem \ref{Th:stab}. Establishing the stability, without the need for continuity and a positive lower bound of the threshold function, allowed us to demonstrate the existence of an exact solution under weaker assumptions than those in Theorem \ref{Th:AFK}, namely without assumptions (A1), (A2), and (A3) (Theorem \ref{Th:exact}). Moreover, Theorem \ref{Th:strong} proved that solutions obtained through this scheme strongly converge to the exact solutions under the specified norms.

In this paper, we exclusively addressed the case where $K$ represents the von Mises model. For future work, it is necessary to investigate whether the proposed scheme can be adapted for cases where $K$ is a general convex set or applied to models such as the Drucker--Prager model \cite{dSNPO08}. While this paper focused on time discretization, exploring fully discretized cases, which involve both time and spatial discretization, is an important next step in numerical computations. Specifically, when applying the finite element method, selecting the appropriate finite element spaces for $v$ and $\sigma$ becomes a pivotal step. Additionally, conducting numerical calculations and comparing them with existing methods and experimental results is essential. Addressing convergence in numerical analysis is crucial. This first requires discussing the regularity of the exact solution, which is currently an open problem.

\bibliographystyle{spmpsci}      
\bibliography{Proj4Plasticity.bib}

\appendix
\section{Explicit Representation of $K(t)$}\label{sec:K}

\setcounter{equation}{0}
\renewcommand{\theequation}{A.\arabic{equation}}
\setcounter{figure}{0}
\renewcommand{\thefigure}{\Alph{section}.\arabic{figure}}
\setcounter{table}{0}
\renewcommand{\thetable}{\Alph{section}.\arabic{table}}

For the case of $d=3$, it is well known in engineering that the von Mises yield surface becomes a super-cylinder. In this appendix, we discuss the general case for $2\le d\in\N$.

We introduce the characterization of the closed set \(K_R \subset \mathbb{R}^{d \times d}\) defined by \(K_R \defeq \{\sigma \in \mathbb{R}^{d \times d} : |\sigma^D| \le R\}\). To achieve this, we define linear, distance-preserving maps \(\Psi_1\) and \(\Psi_2\) that transform \(\mathbb{R}\) and \(\mathbb{R}^{d \times d-1}\) into appropriate subspaces. These maps allow us to decompose elements in \(K_R\) and show the uniqueness of this decomposition. Finally, we prove that for all \(F \in \mathbb{R}^{d \times d}\), the mapping \(K_R\ni\sigma \mapsto |\sigma - F|^2\in\R\) achieves its minimum on \(K_R\) at a specific point, facilitating further analysis of yield conditions. Lastly, we show that \eqref{ineq:2ndP} can be explicitly calculated using \(\Pop{g_n}\) in the form of the second equation of \eqref{implicit}.

First of all, we prepare Assumption \ref{Assump:Psi}.

\begin{Assump}\label{Assump:Psi}
  Maps $\Psi_1:\R\rightarrow\R^{d\times d}$ and $\Psi_2:\R^{d\times d-1}\rightarrow\R^{d\times d}$ satisfy the following conditions.
  \begin{enumerate}
      \item[(i)] $\Psi_1$ and $\Psi_2$ are linear.
      \item[(ii)] $\Psi_1$ and $\Psi_2$ are distance preserving maps, i.e., $\Psi_1$ and $\Psi_2$ satisfy that
      \[
          |\Psi_1(\lambda)|=|\lambda|,\quad 
          |\Psi_2(x)|=|x|
      \]
      for all $\lambda\in\R$ and $x\in\R^{d\times d-1}$.
      \item[(iii)] $\Psi_1(\R) = E_d\R (=\{\mu E_d:\mu\in\R\})$ and 
      $\Psi_2(\R^{d\times d-1}) = \{A\in\R^{d\times d} : \tr A = 0\}$. 
  \end{enumerate}
\end{Assump}

For example, if we define $\Psi_1:\R\rightarrow\R^{d\times d}$ and $\Psi_2:\R^{d\times d-1}\rightarrow\R^{d\times d}$ as follows; for $\lambda\in\R$ and $x=((a_i)_{i=1,\ldots,d-1}, (b_{ij})_{i,j=1,\ldots,d\mbox{ with }i\ne j})\in\R^{d\times d}$,
\[
  \Psi_1(\lambda) = \frac{\lambda}{\sqrt{d}}E_d,\quad 
  \Psi_2(x) = \sum_{i=1}^{d-1} a_i e_i + \sum_{i \ne j}b_{ij}E_{ij},
\]
where $e_i \in \R^{d\times d}$ ($i=1,\ldots,d-1$) is defined by
\begin{align*}
  e_i \defeq \frac{1}{\sqrt{i(i+1)}}\diag(\underbrace{1, \ldots, 1}_{i},-i,0,\ldots,0),
\end{align*}
and $E_{ij}$ ($i, j = 1, \ldots, d$) is defined as the \(d \times d\) matrix where the \((i,j)\)-entry is 1 and all other entries are 0, then $\Psi_1$ and $\Psi_2$ satisfy Assumption \ref{Assump:Psi}. 

\begin{Prop}\label{prop:KR}
  We assume that the maps $\Psi_1:\R\rightarrow\R^{d\times d}$ and $\Psi_2:\R^{d\times d-1}\rightarrow\R^{d\times d}$ satisfy Assumption \ref{Assump:Psi}.
  Let $R>0$ and let $K_R \defeq \{\sigma\in\R^{d \times d}:|\sigma^D| \le R\}$. Then we have that 
  \begin{align}\label{eq:KR}
      K_R = \{\Psi_1(\lambda) + \Psi_2(x) : \lambda\in\R, x\in\R^{d\times d-1}, |x|\le R\}.
  \end{align}
  Furthermore, for all $\sigma\in K_R$, there exist unique $\lambda\in\R$ and $x\in\R^{d\times d-1}$ such that $\sigma = \Psi_1(\lambda) + \Psi_2(x)$.
\end{Prop}

\begin{proof}
  First, we show that 
  \begin{align}\label{eq:decomp}
      R^{d\times d} = \Psi_1(\R) \oplus \Psi_2(\R^{d\times d-1})
      = \{\Psi_1(\lambda) + \Psi_2(x) : \lambda\in\R, x\in\R^{d\times d-1}\}.
  \end{align}
  The map $\Pop{0}:\R^{d\times d}\ni A\mapsto (\tr A/d)E_d\in\R^{d\times d}$ is a linear map and satisfies that $\Pop{0}^2 = \Pop{0}$. Hence, it holds that 
  \[
      \R^{d\times d} = \operatorname{Im}(\Pop{0}) \oplus \operatorname{Ker}(\Pop{0}),
  \]
  where $\operatorname{Im}(\Pop{0})$ and $\operatorname{Ker}(\Pop{0})$ is the image and the kernel of $\Pop{0}$, respectively. Since it holds that 
  \[
      \operatorname{Im}(\Pop{0}) = E_d\R = \Psi_1(\R),\quad
      \operatorname{Ker}(\Pop{0}) = \{A\in\R^{d\times d}:\tr A=0\} = \Psi_2(\R^{d\times d-1}),
  \]
  we obtain \eqref{eq:decomp}.

  We have that 
  \[
      K_R = (\operatorname{Im}(\Pop{0}) \cap K_R) \oplus (\operatorname{Ker}(\Pop{0}) \cap K_R),
  \]
  and, by Assumption \ref{Assump:Psi} (ii) and (iii),
  \begin{align*}
      (\operatorname{Im}(\Pop{0}) \cap K_R) 
      &= E_d\R = \Psi_1(\R),\\
      (\operatorname{Ker}(\Pop{0}) \cap K_R)
      &= \{A\in\R^{d\times d}:\tr A=0, |A^D|\le R\}
      = \{A\in\R^{d\times d}:\tr A=0, |A|\le R\}\\
      &= \{\Psi_2(x):x\in\R^{d\times d-1}, |x|<R\},
  \end{align*}
  which implies that
  \begin{align*}
      K_R 
      &= \Psi_1(\R) \oplus \{\Psi_2(x):x\in\R^{d\times d-1}, |x|<R\}\\
      &= \{\Psi_1(\lambda) + \Psi_2(x) : \lambda\in\R, x\in\R^{d\times d-1}, |x|\le R\}.
  \end{align*}
\end{proof}

\begin{Lem}\label{Lem:min}
  We assume that maps $\Psi_1:\R\rightarrow\R^{d\times d}$ and $\Psi_2:\R^{d\times d-1}\rightarrow\R^{d\times d}$ satisfy Assumption \ref{Assump:Psi}.
  For $R>0$ and $F\in\R^{d\times d}$, the mapping $K_R \ni \sigma \mapsto |\sigma-F|^2 \in \R$ achieves its minimum at $\sigma=\Pop{R}(F)$.
\end{Lem}

\begin{proof}
  First, we consider the case where $F\in K_R$. The mapping $\R^{3\times 3} \ni \sigma \mapsto |\sigma-F|^2 \in \R$ achieves its minimum value of $0$ at $\sigma=F$. Hence, if $F \in K_R$, then the mapping $K_R \ni \sigma \mapsto |\sigma-F|^2 \in \R$ achieves its minimum at $\sigma = F$.

  Next, we consider the case where $F \notin K_R$. By Proposition \ref{prop:KR}, for all $\tau\in K_R$, there exist unique $\mu\in\R$ and $y\in\R^{d\times d-1}$ such that $\tau = \Psi_1(\mu) + \Psi_2(y)$ and $|y| \le R$. 
  \begin{align*}
      |\tau - F|^2 &= \left|\Psi_1(\mu) + \Psi_2(y) - \frac{\tr F}{d}E_d - F^D\right|^2\\
      &= \left|\Psi_1(\mu) - \frac{\tr F}{d}E_d\right|^2 + \left|\Psi_2(y) - F^D\right|^2,
  \end{align*}
  where we have used $\Psi_1(\mu) - (\tr F/d)E_d \in E_d\R$ and $\Psi_2(y) - F^D \in \{A\in\R^{d\times d} : \tr A = 0\}$.
  Here, $|\Psi_1(\mu) - (\tr F/d)E_d|^2$ depends on only $\mu$ and $|\Psi_2(y) - F^D|^2$ depends on only $y$. 
  
  By Assumption \ref{Assump:Psi} (iii), there exist $\lambda\in\R$ and $z\in\R^{d\times d-1}$ such that $\Psi_1(\lambda) = (\tr F/d)E_d$ and $\Psi_2(z) = F^D$. The mapping $\R\ni\mu\mapsto|\Psi_1(\mu) - (\tr F/d)E_d|^2\in\R$ achieves its minimum at $\mu=\lambda$. By Assumption \ref{Assump:Psi} (i), the mapping $\{y\in\R^{d\times d-1}:|y| \le R\}\ni y\mapsto|\Psi_2(y) - F^D|^2 = |\Psi_2(y) - \Psi_2(z)|^2 = |y-z|^2$ achieves its minimum at $y = (R/|z|)z$. 
  
  Hence, if $F\notin K_R$, then the mapping $K_R \ni \sigma \mapsto |\sigma-F|^2 \in \R$ achieves its minimum at 
  \[
      \sigma 
      = \Psi_1(\lambda) + \Psi_2\left(\frac{R}{|z|}z\right)
      = \frac{\tr F}{d}E_d + \frac{R}{|\Psi_2(z)|}\Psi_2(z)
      = \frac{\tr F}{d}E_d + \frac{R}{|F^D|}F^D.
  \]
  Therefore, the mapping $K_R \ni \sigma \mapsto |\sigma-F|^2 \in \R$ achieves its minimum at 
  \[
      \sigma = \left\{\begin{aligned}
          &F && \mbox{if } |F^D|\le R\\
          &\frac{\tr F}{d}E_d + \frac{R}{|F^D|}F^D
          && \mbox{if } |F^D| > R
      \end{aligned}\right.
  \]
  i.e. $\sigma = \Pop{R}(F)$
\end{proof}

From Lemma \ref{Lem:min}, one can obtain the following result.
\begin{Th}\label{th:min}
  Let $\dt>0$, $\sigma_a, \sigma_b, h, p\in\Hs$, $v\in\Vv$, and $g\in L^2(\Omega)$. It holds that 
  \[
      \frac{\sigma_b - \sigma_a}{\dt} 
      \in \E(v) + h - \partial I_K(\sigma_b)
      \Leftrightarrow 
      \sigma_b + p = \Pop{g}(\sigma_a + \dt(\E(v)+h) + p),
  \]
  where $K \defeq \tilde{K}-p$ and 
  \[
      \tilde{K} \defeq \left\{\tau\in\Hs : |\tau^D| \le g \mbox{ a.e. in }\Omega\right\}.
  \]
\end{Th}

\begin{proof}
  We put $F \defeq \sigma_a+\dt(\E(v)+h)$. By the definition of subdifferential, it holds that 
  \begin{align*}
      &
      \frac{\sigma_b - \sigma_a}{\dt} 
      \in \E(v) + h - \partial I_K(\sigma_b) \\
      \Leftrightarrow&
      I_K(\tau) - I_K(\sigma_b) 
      \ge \left(\E(v) + h - \frac{\sigma_b - \sigma_a}{\dt}, \tau - \sigma_b \right) \mbox{ for all }\tau\in\Hs \\
      \Leftrightarrow&
      \sigma_b\in K\mbox{ and }
      0\ge \left(F - \sigma_b, \tau - \sigma_b \right) \mbox{ for all }\tau \in K\\
      \Leftrightarrow&
      \sigma_b\in K\mbox{ and }
      \|\sigma_b - F\|_\Hs \le \|\tau - F\|_\Hs \mbox{ for all }\tau \in K.
  \end{align*}
  Since $\tau\in K = \tilde{K}-p \Leftrightarrow \tau + p\in \tilde{K}$, we have 
  \begin{align*}
      &
      \frac{\sigma_b - \sigma_a}{\dt} 
      \in \E(v) + h - \partial I_K(\sigma_b) \\
      \Leftrightarrow&
      \sigma_b\in K\mbox{ and }
      \|\sigma_b - (F+p)\|_\Hs \le \|\tau - (F+p)\|_\Hs \mbox{ for all }\tau \in \tilde{K}.
  \end{align*}
  Therefore, by Lemma \ref{Lem:min}, 
  \begin{align*}
      \frac{\sigma_b - \sigma_a}{\dt} 
      \in \E(v) + h - \partial I_K(\sigma_b) 
      \Leftrightarrow
      \sigma_b = \Pop{g}(F+p)
      = \Pop{g}(\sigma_a + \dt(\E(v)+h) + p).
  \end{align*}
\end{proof}



\end{document}